\documentclass[11pt,a4paper]{article}
\usepackage{graphicx}
\usepackage{amsfonts,amssymb,amsmath,amsthm}
\usepackage{}

\newcommand{\ep}{\varepsilon}
\newcommand{\R}{\mathbb{R}}

\newtheorem{them}{Theorem}[section]
\newtheorem{lem}[them]{Lemma}

\newtheorem{corol}[them]{Corollary}
\newtheorem{pro}[them]{Proposition}

\newtheorem{Def}[them]{Definition}
\newtheorem{Rema}[them]{Remark}

\newcommand \Hp {H^1_{per}}

\begin{document}
%\large
\title{A variational problem associated with the minimal speed of travelling waves for spatially periodic reaction-diffusion equations}
\author{Xing Liang \thanks{Email: xliang@ustc.edu.cn} \\
Department of Mathematics\\ University of Science and Technology
of China \\ and \\Graduate School
 of Mathematical Sciences\\ University of Tokyo \and Xiaotao Lin\thanks{Email: linxt@ms.u-tokyo.ac.jp; Corresponding author} \\
Graduate School
 of Mathematical Sciences\\ University of Tokyo  \and Hiroshi Matano\thanks{Email: matano@ms.u-tokyo.ac.jp} \\
Graduate School
 of Mathematical Sciences\\ University of Tokyo}

\date{}
\maketitle
\begin{abstract}
 We consider the equation $u_t=u_{xx}+b(x)u(1-u),$ $x\in\mathbb R,$
 where $b(x)$ is a nonnegative measure on $\mathbb R$ that is
 periodic in $x.$ In the case where $b(x)$ is a smooth periodic
 function, it is known that there exists a travelling wave with
 speed $c$ for any $c\geq c^*(b),$ where $c^*(b)$ is a certain
 positive number depending on $b.$ Such a travelling wave is often
 called a \lq\lq pulsating travelling wave" or a \lq\lq periodic
 travelling wave", and $c^*(b)$ is called the \lq\lq minimal
 speed". In this paper, we first extend this theory by showing the
 existence of the minimal speed $c^*(b)$ for any nonnegative
 measure $b$ with period $L.$ Next we study the question of maximizing $c^*(b)$
 under the constraint $\int_{[0,L)}b(x)dx=\alpha L,$ where $\alpha$ is an arbitrarily given constant.
 This question  is closely related to the problem studied by mathematical ecologists in late 1980's but its answer
 has not been known. We answer this question by proving that
 the maximum is attained by periodically arrayed Dirac's delta
 functions $\alpha L\sum_{k\in\mathbb Z}\delta(x+kL).$
\end{abstract}

{\section{Introduction}} Travelling wave solutions describe a wide
class of phenomena in combustion physics, chemical kinetics,
biology and other natural sciences. {}From the physical point of
view, travelling waves usually describe transition processes.
Transition from one equilibrium to another is a typical case,
although more complicated situations may arise. Since the
classical paper by Kolmogorov, Petrovsky and Piskunov in 1937,
travelling wave solutions have been intensively studied. For
example, the monograph of Volpert, Volpert and Volpert \cite{s2}
provides a comprehensive discussion on this subject.

{}From the ecological point of view, travelling waves typically
describe the expansion of the territory of a certain species,
including, in particular, the invasion of alien species in given
habitat. This kind of process may occur in both homogeneous and
heterogeneous media. Models for biological invasions in spatially
periodic environments were first introduced by Shigesada {\it{et
al}}. in dimensions 1 and 2 (see \cite{s4,s5,s6}). More precisely,
they considered spatially segmented habitats where favorable and
less favorable (or even unfavorable) zones appear alternately and
analyzed how the pattern and scale of spatial fragmentation affect
the speed of invasions. In their study, the spatial fragmentation
was typically represented by step functions which take two
different values periodically. Mathematically, their analysis was
partly unrigorous as it relied on formal asymptotics of the
travelling wave far away from the front.

Berestycki, Hamel \cite{BH} and Berestycki, Hamel, Roques
\cite{s1} extended and mathematically deepened the work of
Shigesada \emph{et al}. significantly, by dealing with much more
general equations of the form $u_t=\nabla \cdot ((A(x)\nabla
u))+f(x,u)$ in $\mathbb R^n$ with rather general smooth periodic
coefficients and by developing various mathematical techniques to
study the effect of environmental fragmentation rigourously.

Among other things, they proved that, under certain assumptions on
the coefficients, there exists $c^*>0$ such that the equation has
a pulsating travelling wave solution if and only if $c\geq c^*.$
Furthermore, they showed that the minimal speed $c^*$ is
characterized by the following formula:
\[
\ c^*=\min\{c>0\,|\, \exists \lambda>0 \hbox{ such that
}\mu(c,\lambda)=0 \},
\]
where $\mu(c,\lambda)$ is the principal eigenvalue of a certain
elliptic operator associated with the linearization  of the
travelling wave far away from the front. A more detailed account
of this result will be stated in Subsection \ref{basicnotions} in
the special context of our problem.

By using a totally different approach Weinberger \cite{s8} also
proved the existence of the minimal speed $c^*$ of pulsating
travelling waves in a more abstract framework. His method relies
on the theory of monotone operators and is a generalization of his
earlier work \cite{Wein1982} to spatially periodic media.

It is important to note that, as far as one-dimensional diffusion
equations are concerned, the minimal speed $c^*$ coincides with
the so-called spreading speed for a large class of monostable
nonlinearities. Here the \lq\lq spreading speed'' roughly means
the asymptotic speed of an expanding front that starts from a
compactly supported initial data (see Definition
\ref{defofspreadingspeed} for details). An early study of
spreading speeds in multi-dimensional spaces can be found in
\cite{AroWei1,AroWei2,Wein1982}. Weinberger \cite{s8} then studied
the spreading speeds of order-preserving monostable mappings and
applied the results to spatially periodic reaction diffusion
equations and lattice systems.

Berestycki, Hamel and Nadirashvili \cite{BHN} also studied the
spreading speed of reaction diffusion equation in a very general
periodically fragmented environment where both the coefficients
and the domain itself are periodic. Recently, the same authors
\cite{BHN2} studied the spreading speed of reaction-diffusion
equations with constant coefficients, but in very general domains
which are not necessarily periodic.

In this paper we consider the following equation
\begin{equation}\label{firstequation}
u_t=u_{xx}+b(x)u(1-u),\,\,\,x\in\mathbb R,
\end{equation}
where $b(x)$ is either a smooth function or a measure satisfying
$b(x)\geq 0$ and $b(x+L)\equiv b(x),\,x\in\mathbb R,$ for some
$L>0.$

By the above-mentioned work \cite{BH,s1,BHN,s8}, the minimal speed
$c^*$ of travelling waves is well-defined at least as far as
$b(x)$ is a smooth function, and it coincides with the spreading
speed. We denote this minimal speed by $c^*(b).$ The goal of the
present paper is to consider the variational problem
\[
\ \underset{b}{\hbox{Maximize }}c^*(b)
\]
under the constraint
\begin{equation}\label{intergral}
\ \int_{[0,L)} b(x)dx=\alpha L,
\end{equation}
where $\alpha>0$ is an arbitrarily given constant. In other words
we want to find out whether or not there exists an optimal $b(x)$
that gives the fastest spreading speed. We will show that the
maximum of $c^*(b)$ does indeed exist but that it is not attained
by any smooth function $b(x)$ but by a measure which is composed
of periodically arrayed Dirac's delta functions.

In order to study the above problem, we have to consider the
equation
 \begin{equation}\label{Mainequation}u_t=u_{xx}+\bar{b}(x)u(1-u),\,\,\,x\in\mathbb R,\end{equation} where $\bar{b}$ is a nonnegative measure satisfying \eqref{integral}.
We introdcue two important quantities  $c^*(\bar{b})$ and
$c_e^*(\bar{b})$.  The former denotes the minimal speed of
travelling waves for equation \eqref{Mainequation}. The latter is
a quantity associated with a generalized eigenvalue problem (see
Definition \ref{Defofcestar} below). It has been known that
$c^*(b)=c_e^*(b)$ if $b$ is smooth. As we will see later,
 $c^*(\bar{b})=c_e^*(\bar{b})$ even if $\bar{b}$ is a measure (Theorem \ref{cstarbequalcestarb}
). We will then show in Theorem \ref{cestarhlargerthancestarb}
that the maximum of $c^*(\bar{b})$ is attained by
\begin{equation}\label{h}
\bar{b}(x)=h(x):=\alpha L\sum_{k \in\mathbb Z}
\delta\Big(x-\big(k+\frac{1}{2}\big)L\Big),
\end{equation}
where $\delta(x)$ is Dirac's delta function.

\fbox{Change begins}

Strictly speaking, we have to distinguish the travelling wave
speeds in the positive direction and those in the negative
direction (see Definition \ref{Definitionoftravellingwave}). The
above mentioned quantities $c^*(\bar{b})$ and $c_e^*(\bar{b})$ are
associated with travelling waves in the positive direction.
However, as we will see later in Theorem \ref{cstarbequalcestarb},
the two speeds - positive and negative - are always equal,
therefore no ambiguity occurs by not specifying the direction of
the travelling wave.

\fbox{Change ends}

This paper is organized as follows. In Section 2, we introduce
basic notations and state the main results. In Section 3, we prove
the well-posedness of equation \eqref{Mainequation}. In Section 4,
we consider a generalized eigenvalue problem associated with
equation \eqref{Mainequation}. We show that $c^*(b)$ is bounded
when $b$ is smooth and that $c_e^*(\bar{b})$ is bounded when
$\bar{b}$ varies in a certain set of measures. In Section 5, we
show that the minimal speed $c^*(\bar{b})$ of travelling waves
exists when $\bar{b}$ is a measure, and it coincides with
$c_e^*(\bar{b})$ and also with the spreading speed.  Then we
complete the proof of Theorem \ref{cestarhlargerthancestarb}. In
Section 6, we prove the lemmas on equicontinuity of the solutions
of Cauchy problem (see equation \eqref{equationinsection3} below).
These lemmas are used in Sections 3 and 5.

\section{Notation and main results}

\subsection{Basic notation}\label{basicnotions}

In this subsection we introduce some notation and recall some
known results which will be used later.

In what follows we fix constants $L>0$ and $\alpha>0.$ Let
$\Lambda(\alpha)$ be the set defined by
\[
\ \Lambda(\alpha):=\{b(x) \in C^1(\mathbb R)\,|\, b(x)\geq
0,b(x)=b(x+L) \,~and~ \int_{[0,L)} b(x)dx=\alpha L \}.
\]

\begin{Def}\label{DefofLamdaBarAlpha}
$\overline{\Lambda}(\alpha)$ is defined to be the sequential
closure of $\Lambda(\alpha)$ in the space of distribution on
$\mathbb{R}.$ More precisely,
$\bar{b}\in\overline{\Lambda}(\alpha)$ if and only if there exists
a sequence $\{b_n\}_{n=1}^{\infty}$ in $\Lambda(\alpha)$ such that
\begin{equation}\label{DefinitionOfLamdabarAlpha}
\
\int_{\mathbb{R}}\bar{b}(x)\eta(x)dx=\lim_{n\to\infty}\int_{\mathbb{R}}b_n(x)\eta(x)dx
\end{equation}
for any test function $\eta\in C_0^\infty(\mathbb{R}),$ where the
left-hand side of \eqref{DefinitionOfLamdabarAlpha} is a formal
integration representing the dual product $<\bar{b},\eta>.$ In
this sense, we denote that $b_n\to\bar{b}$ in the weak \,$*$
sense.
\end{Def}

Since each $b_n$ is positive, $\bar{b}$ is a nonnegative
distribution. Consequently, $\bar{b}$ is a Borel measure on
$\mathbb{R}.$ Therefore, \eqref{DefinitionOfLamdabarAlpha} holds
for every $\eta\in C_0(\mathbb{R}).$

In what follows, we will not distinguish the measure $\bar{b}$ and
its density function $\bar{b}(x),$ as long as there is no fear of
confusion. Thus we will often use expression as in the left-hand
side of \eqref{DefinitionOfLamdabarAlpha}.

We will also note that, since $b_n(x)$ is $L-periodic,$
$\bar{b}(x)$ is also $L-periodic$ in the following sense:
\begin{equation}\label{L-periodicity}
\ \int_\mathbb{R}\bar{b}(x)
\eta(x+L)dx=\int_{\mathbb{R}}\bar{b}(x)\eta(x)dx
\end{equation}
for $\eta\in C_0(\mathbb{R}).$

The following lemmas will be useful later:
\begin{lem}\label{lem:A}
Let $\{b_n\}\subset\Lambda(\alpha)$ be a sequence converging to
some $\bar{b}\in\overline{\Lambda}(\alpha)$ in the weak \,$*$
sense. Let $\eta(x)$ be a continuous function on $\mathbb{R}$
satisfying
\begin{equation}\label{sum-eta}
\sum_{k=-\infty}^{\infty}\max_{0\leq x \leq L}|\eta(x+kL)|<\infty.
\end{equation}
Then $\eta$ is $\bar{b}$-integrable on $\mathbb{R}$ and the
following hold:
\begin{equation}\label{bn-to-b-R}
\lim_{n\to\infty}\int_{\mathbb{R}}b_n(x)\eta(x)dx=\int_{\mathbb{R}}\bar{b}(x)
\eta(x)dx,
\end{equation}
\begin{equation}\label{sum-eta}
\int_{\mathbb{R}}\bar{b}(x)|\eta(x)|dx \leq \alpha L
\sum_{k=-\infty}^{\infty}\max_{0\leq x \leq L}|\eta(x+kL)|.
\end{equation}
\end{lem}

\begin{proof}
For each integer $M>0$, we define a cut-off function $q_M(x)$ by
\[
q_M(x)=\left\{
\begin{array}{ll}
0 \quad & \hbox{for}\ \ |x|>(M+1)L\\
1 \quad & \hbox{for}\ \ |x|\leq ML\\
M+1-L^{-1}|x| \quad & \hbox{otherwise}.
\end{array}\right.
\]
Then, since $q_M(x)|\eta(x)|$ is continous and compatctly
supported, we have
\begin{equation}\label{limit-qM}
\int_{\mathbb{R}}\bar{b}(x)q_M(x)|\eta(x)|dx
=\lim_{n\to\infty}\int_{\mathbb{R}}b_n(x)q_M(x)|\eta(x)|dx.
\end{equation}
Note also that, since each $b_n$ belongs to $\Lambda(\alpha)$,
\[
\begin{array}{ll}
\displaystyle \int_{\mathbb{R}}b_n(x)q_M(x)|\eta(x)|dx &
\displaystyle =\sum_{k=-\infty}^{\infty}
\int_0^L b_n(x)q_M(x+kL)|\eta(x+kL)|dx \vspace{4pt}\\
& \displaystyle \leq \alpha L \sum_{k=-\infty}^{\infty}
\max_{0\leq x \leq L}|\eta(x+kL)|.
\end{array}
\]
This and \eqref{limit-qM} imply
\[
\int_{\mathbb{R}}\bar{b}(x)q_M(x)|\eta(x)|dx\leq \alpha L
\sum_{k=-\infty}^{\infty}\max_{0\leq x \leq L}|\eta(x+kL)|
\]
for $M=1,2,3,\cdots$.  Letting $M\to\infty$ and applying the
monotone convergence theorem, we obtain \eqref{sum-eta}. Hence
$\eta$ is $\bar b$-integrable on $\mathbb{R}$.

Next we observe that
\[
\begin{array}{ll}
\displaystyle
\Big|\int_{\mathbb{R}}b_n(x)\eta(x)dx-\int_{\mathbb{R}}b_n(x)q_M(x)\eta(x)dx\Big|
& \displaystyle
\leq \int_{\mathbb{R}}b_n(x)\big(1-q_M(x)\big)|\eta(x)|dx\vspace{5pt}\\
& \displaystyle \leq \alpha L \sum_{|k|\geq M} \max_{0\leq x \leq
L}|\eta(x+kL)|
\end{array}
\]
The assertion \eqref{bn-to-b-R} easily follows from this and
\eqref{limit-qM}.  The lemma is proven.
\end{proof}

\begin{lem}
Let $\{b_n\}\subset\Lambda(\alpha)$ be a sequence converging to
some $\bar{b}\in\overline{\Lambda}(\alpha)$ in the weak \,$*$
sense. Let $\eta(x)$ be an $L$-periodic continuous function on
$\mathbb R$. Then
\begin{equation}\label{bn-to-b-L}
\lim_{n\to\infty}\int_{[0,L)}b_n(x)\eta(x)dx=\int_{[0,L)}\bar{b}(x)
\eta(x)dx,
\end{equation}
\end{lem}

\begin{proof}
By the bounded convergence theorem, we have
\begin{equation}\label{[0,L)}
 \int_{[0,L)} \bar{b}(x)\eta(x)dx =
 \lim_{\ep\to 0} \int_{\R} \bar{b}(x)p_{\ep}(x)\eta(x)dx,
\end{equation}
where $p_\ep$ is a cut-off function defined by
\[
p_\ep(x)=\left\{
\begin{array}{ll}
0 \quad & \hbox{for}\ \ x\in(-\infty,-\ep)\cup [L,\infty)\\
\ep^{-1}(x+\ep) \quad & \hbox{for}\ \ x\in [-\ep,0)\\
1 \quad & \hbox{for}\ \ x\in[0, L-\ep)\\
\ep^{-1}(L-x) \quad & \hbox{for}\ \ x\in [L-\ep,L).
\end{array}\right.
\]
On the other hand, by the $L$-periodicity of $b_n$ and $\eta$, we
have
\[
 \int_{[0,L)} b_n(x)\eta(x)dx =
 \int_{\R} b_n(x)p_{\ep}(x)\eta(x)dx
 \quad\ \ (n=1,2,3,\cdots)
\]
for any $0<\ep\leq L/2$. It follows that
\[
\lim_{n\to\infty} \int_{[0,L)} b_n(x)\eta(x)dx = \lim_{n\to\infty}
\int_{\R} b_n(x)p_{\ep}(x)\eta(x)dx = \int_{\R}
\bar{b}(x)p_{\ep}(x)\eta(x)dx.
\]
Combining this and \eqref{[0,L)}, we obtain the desired identity.
The lemma is proven.
\end{proof}

We remark that, because of the $L-$periodicity of the measure
$\bar{b}\in\overline{\Lambda}(\alpha),$ we can regard $\bar{b}$ as
functional in the space $C^*(\mathbb{R}/L\mathbb{Z}),$ which we
denote by $[\bar{b}],$ in the following sense:
\[
\
\int_{\mathbb{R}/L\mathbb{Z}}[\bar{b}](x)[\eta](x)dx=\int_{[0,L)}\bar{b}(x)\eta(x)dx
\] for any $[\eta]\in C(\mathbb{R}/L\mathbb{Z}).$ Here $\eta(x)$
is  the periodic continuous function on $\mathbb{R}$ associated
with $[\eta].$ Furthermore, by \eqref{bn-to-b-L}, we have if
$b_n\to\bar{b}$ in the weak \,$*$ sense, then $[b_n]\to[\bar{b}]$
in the weak \,$*$ sense in $C^*(\mathbb{R}/L\mathbb{Z}).$

%\begin{Rema} Any element $\bar{b}$ of $\overline{\Lambda}(\alpha)$
%is periodic with respect to $x$ in the sense that for any $g \in
%C[0,L]$, we have
%\[
%\ \int_{[0,L)} \bar{b}(x+nL)g(x)dx =\int_{[0,L)}
%\bar{b}(x)g(x)dx,\,\,\,\,\forall n\in\mathbb Z.
%\]
%\end{Rema}
% For example,
%\begin{equation}\label{h}
%\ h(x)=\alpha L\sum_{k \in\mathbb Z}
%\delta\Big(x-\big(k+\frac{1}{2}\big)L\Big)
%\end{equation}
%is an element of $\overline{\Lambda}(\alpha)$, but it is not an
%element of $\Lambda(\alpha)$. Here $\delta(x)$ is Dirac's delta
%function, which is defined by
%\[
%\ \int_{-\infty}^{+\infty}\delta(x)g(x)dx=g(0),\;\,\;\forall g\in
%C_0^\infty(\mathbb R).
%\]

We consider the following reaction-diffusion equation:
\[
\ u_t=u_{xx}+\bar{b}(x)u(1-u), \,\,\,x\in \mathbb R,
\]
where $\bar{b} \in \overline{\Lambda}(\alpha)$. Here we have not
specified the range of $t,$ but what we have typically in mind is
either $t>0$ or $t\in \mathbb R.$ Obviously
$\Lambda(\alpha)\subset \overline{\Lambda}(\alpha),$ therefore
\eqref{Mainequation} is a generalization of (1).

\begin{Def}
Let $I\subset \mathbb R$ be any open interval. A continuous
function $u(x,t):\mathbb R\times I\to \mathbb R$ is called a
{\textbf{weak solution}} of \eqref{Mainequation} for $t\in I$ (or
a {solution} in the weak sense) if for any $\eta(x,t)\in
C_0^\infty(\mathbb R\times I),$
\[
\ -\int_{\mathbb R\times I}u \eta_t\,dxdt=\int_{\mathbb R\times
I}\big(u\eta_{xx}+\bar{b}(x)u(1-u)\eta\big)\, dxdt,
\]
where the second integral on the right-hand side is understood in
the following sense:
\begin{equation}\label{baslinearfunctional}
\ \int_{\mathbb R\times I}\bar{b}(x)u(1-u)\eta\,
dxdt=\int_I\Big(\sum_{k\in\mathbb
Z}\int_{[kL,(k+1)L)}\bar{b}(x)u(1-u)\eta\, dx\Big) dt.
\end{equation}
\end{Def}

We next consider the following Cauchy problem:
\begin{equation}\label{equationinsection3}
\left\{\begin{array}{ll} u_t=u_{xx}+\bar{b}(x)u(1-u)
\quad\ \ &(x\in\mathbb R,\;t>0),\vspace{5pt}\\
u(x,0)=u_0(x)\geq 0\quad\ \ &(x\in\mathbb R),
\end{array}\right.
\end{equation}
where $u_0\in C(\mathbb R)\cap L^\infty(\mathbb R).$
\begin{Def}
A continuous function $u(x,t):\mathbb R\times(0,\infty)\to \mathbb
R$ is called a {\textbf{mild solution}} of
\eqref{equationinsection3} if
\[
\ \lim_{t\searrow 0}u(x,t)=u_0(x)\,\hbox{ for any }x \in \mathbb R
\]
and if it can be written as
\[
\begin{split}
 u(x,t)=&\int_{\mathbb
R}G(x-y,t)u_0(y)dy  \\ &+\int_0^t\int_{\mathbb
R}G(x-y,t-s)\bar{b}(y)u(y,s)\big(1-u(y,s)\big)\,dyds,
\end{split}
\]
where
\[
G(x,t):=\frac{1}{\sqrt{4\pi t}}\exp\big(-\frac{x^2}{4t}\big).
\]
\end{Def}
As we will show in Section 3, a mild solution of
\eqref{equationinsection3} always exists for any
$\bar{b}\in\overline{\Lambda}(\alpha)$ and $u_0\in C(\mathbb
R)\cap L^\infty(\mathbb R)$ with $u_0\geq 0,$ and it is unique and
is a weak solution.

It is easily seen that, if $u(x,t)$ is a mild solution of
\eqref{equationinsection3}, then for any constant $\tau\geq 0,$
$u(x,x+\tau)$ is a mild solution of \eqref{equationinsection3}
with initial data $u(x,\tau)$ (see Remark
\ref{RemarkForMildSolution}).

We call a function $u(x,t)$ on $\mathbb{R}\times\mathbb{R}$ a mild
solution for $t\in\mathbb{R}$ if, for any $\tau\in\mathbb{R},$
$u(x,t+\tau)$ is mild solution of \eqref{equationinsection3} with
initial data $u_0(x)=u(x,\tau).$

\begin{Def}\label{Definitionoftravellingwave}
A mild solution $u(x,t)$ of \eqref{Mainequation} for $t\in\mathbb
R$ is called a {\textbf{travelling wave}} solution (in the
\textbf{positive} direction) if $0\leq u(x,t)\leq 1$ and if there
exists a constant $T>0$ such that
\[
\ u(x-L,t)=u(x,t+T)\,\,\,\,\,\,\,\,\,\,\,\,\hbox{ for }
(x,t)\in\mathbb R\times\mathbb R,
\]
\[ \ \lim_{x  \to -\infty}u(x,t)=1,~~\lim_{x
\to +\infty}u(x,t)=0 \quad \hbox{locally uniformly in }\,
t\in\mathbb R.
\]It is called a \textbf{travelling wave} solution (in the
\textbf{negative} direction) if
\[
\ u(x+L,t)=u(x,t+T)\,\,\,\,\,\,\,\,\,\,\,\,\hbox{ for }
(x,t)\in\mathbb R\times\mathbb R,
\]
\[ \ \lim_{x  \to -\infty}u(x,t)=0,~~\lim_{x
\to +\infty}u(x,t)=1 \quad \hbox{locally uniformly in
}\,t\in\mathbb R.
\]
\end{Def}
\fbox{Change begins}\begin{Rema}\label{Remarkfor2directionTW} In
what follows, unless otherwise specified, by a travelling wave we
usually mean the one in the positive direction.
\end{Rema}\fbox{Change ends}
Here we call the quantity $c:=L/T$ the {\it speed} (or the
\emph{average speed} or the {\it effective speed}) of the
travelling wave solution $u(x,t)$.

Berestycki and Hamel \cite{BH}, Berestycki, Hamel and Rogues
\cite{s1} and Weinberger \cite{s8} established the existence of
the minimal speed of travelling wave solutions for general
monostable nonlinearities $f(x,u)$ satisfying certain conditions,
and they also gave an eigenvalue characterization of the minimal
speed (see (\ref{spreadingspeed}) below). Here, for
$f(x,u)=b(x)u(1-u)$, if $b\in\Lambda(\alpha),$ it is not difficult
to see that $f(x,u)$ satisfies the assumptions in \cite{BH,s1,s8}.
So for equation \eqref{firstequation}, we know that for any $b \in
\Lambda(\alpha)$, there exists the minimal travelling wave speed $
c^*(b)>0$ in the following sense:
\[
\ \left\{%
\begin{array}{ll}
    c\geq c^*(b)\,\,\,\,\,\,\,\,\Rightarrow & \hbox{There exists travelling wave with speed $c$\,;} \\
    0\leq c<c^*(b)\Rightarrow & \hbox{No travelling wave with speed $c$ exists.} \\
\end{array}%
\right.
\]
To be more precise, $c^*(b)$ is defined to be the minimal
travelling wave speed in the positive direction. As mentioned in
Remark \ref{Remarkfor2directionTW}, one can also define the
minimal travelling wave speed in the negative direction, which one
may call $\tilde{c}^*(b).$ As we will explain in Theorem
\ref{cstarbequalcestarb}, we always have $c^*(b)=\tilde{c}^*(b)$,
even when $b(x)$ is not symmetric, therefore we do not need to
distinguish the two minimal wave speeds.

As we have mentioned earlier, $c^*(b)$ also coincides with the
so-called \lq\lq spreading speed" of expanding fronts for
\eqref{Mainequation}. Here, we define the spreading speed as
follows:
\begin{Def}\label{defofspreadingspeed}
A quantity $c^{**}(\bar{b})>0$ is called the \textbf{spreading
speed} (in the \textbf{positive} direction) if for any nonnegative
initial data $u_0\not\equiv 0$ with compact support, the mild
solution $u(x,t,u_0)$ of \eqref{Mainequation} satisfies that
\begin{enumerate}
\item[{\rm(i)}] $\lim\limits_{t\to \infty}u(x,t,u_0)=0  $\,\,
uniformly in $\{x>ct\}$ \,\,\,\,if \,\,$c> c^{**}(\bar{b})$,
\item[{\rm(ii)}] $\lim\limits_{t\to \infty}u(x,t,u_0)=1  $\,\,
uniformly in $\{0<x<ct\}$ \,\,\,\,if\,\, $0<c< c^{**}(\bar{b})$.
\end{enumerate}
A quantity $\tilde{c}^{**}(\bar{b})>0$ is called the
\textbf{spreading speed} (in the \textbf{negative} direction) if
for any nonnegative initial data $u_0\not\equiv 0$ with compact
support, the mild solution $u(x,t,u_0)$ of \eqref{Mainequation}
satisfies that
\begin{enumerate}
\item[{\rm(i)}] $\lim\limits_{t\to \infty}u(x,t,u_0)=0  $
uniformly in $\{x<-ct\}$ \,\,if \,\,$c> \tilde{c}^{**}(\bar{b})$,
\item[{\rm(ii)}] $\lim\limits_{t\to \infty}u(x,t,u_0)=1  $
uniformly in $\{-ct<x<0\}$ \,\,if\,\, $0<c<
\tilde{c}^{**}(\bar{b})$.
\end{enumerate}
\end{Def}
\fbox{Change begins}
\begin{Rema}
Basic properties of the spreading speed in general periodic
environments are studied in \cite{BH,s1,BHN,s8}. It is known, at
least for smooth $b$, that the spreading speed $c^{**}(b)$
coincides with the minimal wave speed $c^*(b)$ (and, similarly,
$\tilde{c}^{**}(b)=\tilde{c}^*(b)$). As we will show later in
Theorem \ref{cstarbequalcestarb}, we have
$c^{**}(\bar{b})=\tilde{c}^{**}(\bar{b})$ for any
$\bar{b}\in\overline{\Lambda}(\alpha).$
\end{Rema}
It is known that in the \lq\lq leading edge", namely the area
where $u\approx 0,$ we have the asymptotic expression
\begin{equation}\label{AsymptoticExpression}
u(x,t)\sim e^{-\lambda(x-ct)}\psi(x),
\end{equation} where $\psi(x+L)\equiv \psi(x)>0,$ and $\lambda>0$
is some constant. Substituting \eqref{AsymptoticExpression} into
equation (\ref{firstequation}), we obtain the identity
\begin{equation}\label{IdentityForB}
-\psi''(x)+2\lambda\psi'(x)-b(x)\psi(x)=(\lambda^2-\lambda c
)\psi(x).
\end{equation}
This observation motivates us to introduce the following operator,
which generalizes the operator on the left-hand side of
\eqref{IdentityForB} to the case where $b(x)$ is a measure:
\begin{Def} For $\bar{b}\in\overline{\Lambda}(\alpha),$ we define
an (unbounded) operator $-L_{\lambda,\bar{b}}$ on the Banach space
$\{\psi\in C(\mathbb{R})\,|\,\psi(x)=\psi(x+L)\}$ with
$\|\psi\|=\max_{x\in \mathbb{R}}|\psi(x)|$ as follows:
\[
\
-L_{\lambda,\bar{b}}\psi(x)=-\psi''(x)+2\lambda\psi'(x)-\bar{b}(x)\psi(x).
\] Here the derivatives are understood in the \lq\lq weak sense" by which we mean that
$-L_{\lambda,\bar{b}}\psi=g$ if and only if, for any $\varphi\in
C_0^{\infty}(\mathbb{R}),$
\[
\
\int_\mathbb{R}(-\varphi''-2\lambda\varphi'-\bar{b}\varphi)\,\psi\,
dx=\int_{\mathbb{R}}\varphi g \,dx.
\]
%We know that $D(-L_{\lambda,\bar{b}})\neq\emptyset$ since $0\in
%D(-L_{\lambda,\bar{b}}).$
\end{Def}
 \fbox{Change ends}

%
% where $\mu(c,\lambda,b)$ is the principal eigenvalue of the
%operator
%\begin{equation}\label{EigenvalueProblem}
%-L_{\lambda,c,b}\psi=-\psi''(x)+2\lambda\psi'(x)-(\lambda^2-\lambda
%c +b(x))\psi(x)
%\end{equation} under the periodicity condition $\psi(x+L)\equiv\psi(x)$.
%
\begin{Def}\label{definitionofprincipaleigenvalue}
For $\bar{b}\in\overline{\Lambda}(\alpha)$, $\lambda>0,$ we call
$\mu(\lambda,\bar{b})$ the {\textbf{principal eigenvalue}} of the
operator
$-L_{\lambda,\bar{b}}:=-\psi''(x)+2\lambda\psi'(x)-\bar{b}(x)\psi(x)$
if there exists a positive continuous function $\psi$ with
$\psi(x)\equiv\psi(x+L)$ satisfying
\begin{equation}\label{identityofmulambdabbar}
\
-\psi''(x)+2\lambda\psi'(x)-\bar{b}(x)\psi(x)=\mu(\lambda,\bar{b})\psi
\end{equation} in the weak sense. Here $\psi$ is called the
\textbf{principal eigenfunction}.
\end{Def}
For the principal eigenvalue, the following proposition holds. We
will prove it in Section \ref{pulsating travelling waves}
\fbox{Change begins} by converting the problem
\eqref{identityofmulambdabbar} into a more regular eigenvalue
problem for a compact positive operator. \fbox{Change ends}

\begin{pro}\label{principaleigenvalueisunique}
For any $\bar{b}\in\overline{\Lambda}(\alpha)$ and $\lambda>0,$
the principal eigenvalue $\mu(\lambda,\bar{b})$ exists, and it is
unique and simple.
\end{pro}
\fbox{Change begins}We also note that the principal eigenfunction
of \eqref{identityofmulambdabbar} belongs to
$H_{loc}^1(\mathbb{R}),$ as we will see in Subsection 4.2. Now
observe that $\lambda^2-\lambda c$ in \eqref{IdentityForB} is a
constant and that $\psi>0.$ Therefore, \eqref{IdentityForB}
implies that $\mu(\lambda,b)=\lambda^2-\lambda c$ if $b$ is
smooth. In view of this, we define the following quantities when
$\bar{b}$ is a general measure. \fbox{Change ends}
\begin{Def}\label{Defofcestar}
We define the minimal speed  in the positive direction
$c^*(\bar{b})$ and a related value $c_e^*(\bar{b})$ as follows:
\[
\begin{split}
 c^*(\bar{b})&:=\inf\{c>0\,|\,\hbox{traveling wave in the positive direction with speed $c$ exists}\}\\
 c_e^*(\bar{b})&:=\inf\{c>0 \,| \,\exists \lambda>0 \hbox{ such that }\mu(\lambda,\bar{b})=\lambda^2-\lambda c \}.
\end{split}
\]
Similarly, the minimal speed in the negative direction
$\tilde{c}^*(\bar{b})$ and $\tilde{c}_e^*(\bar{b})$ as
\[
\begin{split}
 \tilde{c}^*(\bar{b})&:=\inf\{c>0\,|\,\hbox{traveling wave in the negative direction with speed $c$ exists}\}\\
 \tilde{c}_e^*(\bar{b})&:=\inf\{c>0 \,| \,\exists \lambda>0 \hbox{ such that }\mu(-\lambda,\bar{b})=\lambda^2-\lambda c \}.
\end{split}
\]
\end{Def}

\fbox{Change begins}

If $b(x)$ is a smooth nonnegative function, then by the results of
\cite{BH,s1,BHN} or by those of \cite{s8}, the following
proposition holds.
\begin{pro}[\cite{BH,s1,BHN}, \cite{s8}]\label{cstarbequalscestarbforsmoothb}
For smooth $b\in\Lambda(\alpha)$, we have
\begin{equation}\label{spreadingspeed}
c^{**}(b)=c^*(b)=c_e^*(b),~~~~~
\tilde{c}^{**}(b)=\tilde{c}^*(b)=\tilde{c}_e^*(b).
\end{equation}
\end{pro}

\fbox{Change ends}

 We will see later that conclusion of above proposition remain to
 hold when $b$ is a measure.
\subsection{Main results}
We are now ready to present our main results.

It is shown in \cite{s1} that for any $b \in \Lambda(\alpha)$, we
have $c^*(b)\geq 2\sqrt{\alpha}=c^*(\alpha).$ Our first theorem
gives an upper bound on $c^*(b):$
\begin{pro}\label{boundedness}
For any $b \in \Lambda(\alpha)$,
\begin{equation}\label{boundednessinequality}
2\sqrt{\alpha}\leq c^*(b)\leq 2\sqrt{\alpha+\alpha^2L^2}.
\end{equation}
Moreover, if $b\not\equiv \alpha,$ then $2\sqrt{\alpha}<c^*(b).$
\end{pro}

As we mentioned above, the inequality $2\sqrt{\alpha}\leq c^*(b)$
is found in \cite{s1}. The main novelty of this theorem is the
upper bound.

While the lower bound in \eqref{boundednessinequality} is sharp
since the equality holds for $b\equiv \alpha,$ it has not been
known whether $c^*(b)$ attains its maximum in $\Lambda(\alpha)$ or
not. The next two theorems shows that $c^*(b)$ does not attain its
maximum in $\Lambda(\alpha)$ but it does in the extended class
$\overline{\Lambda}(\alpha).$
\begin{them}[Minimal speed]\label{cstarbequalcestarb}
For any measure $\bar{b}\in\overline{\Lambda}(\alpha),$ there
exists $c^*(\bar{b})>0$ such that a travelling wave in the
positive direction with speed $c$ exists if and only if\, $c\geq
c^*(\bar{b})$. In other words, $c^*(\bar{b})$ is the minimal
travelling wave speed in the positive direction. Furthermore,
\[
\ c^*(\bar{b})=c_e^*(\bar{b}).
\]
Similarly, the minimal travelling wave speed in the negative
direction $\tilde{c}^*(\bar{b})$ exists and
$\tilde{c}^*(\bar{b})=\tilde{c}_e^*(\bar{b}).$ Furthermore, for
any $\bar{b}\in\overline{\Lambda}(\alpha)$,
\[
\ c^*(\bar{b})=\tilde{c}^*(\bar{b}).
\]

%Moreover, $c^*(\bar{b})$ is also the {\textbf spreading speed} in
%the positive direction of equation \eqref{Mainequation} in the
%sense that, for any nonnegative initial data $u_0\not\equiv 0$
%with compact support, the solution $u(x,t,u_0)$ of
%\eqref{Mainequation} satisfies that
%\begin{enumerate}
%\item[{\rm(i)}] $\lim\limits_{t\to \infty}u(x,t,u_0)=0  $
%uniformly in $x>ct$ \,if\, $c> c^*(b)$.%
%
%\item[{\rm(ii)}] $\lim\limits_{t\to \infty}u(x,t,u_0)=1  $
%uniformly in $0<x<ct$ \,if\, $c< c^*(b)$.%
%
%\end{enumerate}
%Furthermore, $c^*(b)$ is also the minimal travelling wave speed in
%the negative direction and the spreading speed in the negative
%direction.
\end{them}

\begin{them}[Optimal coefficient]\label{cestarhlargerthancestarb}
\[ \ c^*(h)=\sup_{b\in\Lambda(\alpha)}c^*(b)=\max_{\bar{b}\in\overline{\Lambda}(\alpha)}c^*(\bar{b}).\] Moreover, \[ \
c^*(h)>c^*(b) \quad\hbox{ for any }\,b\in\Lambda(\alpha).
\]
\end{them}

\fbox{Change begins}

\begin{them}[Spreading speed]\label{Themforspreadingspeed}
For $\bar{b}\in\overline{\Lambda}(\alpha),$ the spreading speed in
the positive direction $c^{**}(\bar{b})$ and that in the negative
direction $\tilde{c}^{**}(\bar{b})$ exist and
\[
\
c^{**}(\bar{b})=c^*(\bar{b}),\,\,\,\tilde{c}^{**}(\bar{b})=\tilde{c}^*(\bar{b}).
\]Consequently,
\[
\ c^{**}(\bar{b})=\tilde{c}^{**}(\bar{b}).
\]
\end{them}

\fbox{Change ends}

To prove Theorem \ref{cestarhlargerthancestarb}, the following
proposition is important. We will prove it in Subsection
\ref{uniformboundofcestarbbar}.

\begin{pro}\label{Convergenceofcstarbn}
Let $\{b_n\}$ be a sequence in $ \Lambda(\alpha)$ converging to
some $\bar{b}\in\overline{\Lambda}(\alpha)$ in the weak\, $*$
sense. Then
\[
\ c_e^*(\bar{b})=\lim_{n\to\infty}c_e^*(b_n).
\]
\end{pro}

%Since for every $\bar{b}\in\overline{\Lambda}(\alpha)$, we may
%find a sequence $\{b_n\}\in \Lambda(\alpha)$ such that $b_n$
%converges to $\bar{b}$ in the weak $*$ sense. We have a
%corollary:\vskip8pt
%\begin{corol}\label{cestarbbarbounded}
%For any $\bar{b}\in\overline{\Lambda}(\alpha),$
%$2\sqrt{\alpha}\leq c_e^*(\bar{b})\leq2\sqrt{\alpha+\alpha^2L^2}.$
%\end{corol}

%As we stated above, the function $h(x)$ is an element of
%$\overline{\Lambda}(\alpha)$. In this paper, we will also show
%that function $h(x)$ can get the supremum of
%$\{c^*(b),b\in\Lambda(\alpha)\}.$ \vskip8pt

\section{Reaction-diffusion equation with a Borel-measure coefficient}

In this section, we establish the well-posedness of the Cauchy
problem \eqref{equationinsection3}.

\begin{them}\label{wlposedness}
For any given nonnegative initial data $u_0\in C(\mathbb R)\cap
L^\infty(\mathbb R),$ the problem \eqref{equationinsection3} has a
unique mild solution. This mild solution is also a weak solution
and it depends continuously on the initial data in the $L^\infty$
norm.
\end{them}

\begin{proof} First we show that the solution $u(x,t)$ exists in the weak sense. Let $b_n\in\Lambda(\alpha)$ satisfy
$b_n\to \bar{b}$ in the weak\,$*$ sense. Then for any given
initial data $u_0(x),$ the problem
\begin{equation}\label{equationforun}
\left\{\begin{array}{ll}
(u_n)_t=(u_n)_{xx}+b_n(x)u_n(1-u_n)\\
 u_n(x,0)=u_0(x)
\end{array}\right.
\end{equation}
has a classical solution $u_n(x,t)$ for any $n\in \mathbb N.$ By
the comparison principle, $0\leq
u_n(x,t)\leq\max\{\|u_0\|_{L^\infty(\mathbb R)},1\}.$ Hence $u_n
\,(n\in\mathbb N)$ are uniformly bounded. Let $0<t_1<t_2$ be two
positive numbers. By Lemma \ref{Equicontinuity} which we will
prove in Section 6, the family of solutions
$\{u_n(x,t),x\in\mathbb R,t\in[t_1,t_2]\}$ are uniformly
equicontinuous with respect to $x$ and $t$. Here the modulus of
equicontinuity may depend on $t_1$ and $t_2$. Applying
Arzela-Ascoli theorem, we can get a subsequence, which we still
denote by $\{u_n(x,t)\}$ that converges uniformly in
$(x,t)\in[-M,M]\times[t_1,t_2]$ for every $M>0$ and $0<t_1<t_2.$
The limit function $u(x,t)=\lim_{n\to\infty}u_n(x,t)$ is defined
for every $(x,t)\in \mathbb R\times(0,+\infty)$ and satisfies
\begin{equation}\label{equationwithnonsmoothb}
u_t=u_{xx}+\bar{b}(x)u(1-u)
\end{equation}
in the weak sense. Next we show that $u$ satisfies
\begin{equation}\label{initial condition}
\lim_{t\searrow 0}u(x,t)=u_0(x)\quad\hbox{ for any } x \in\mathbb
R.
\end{equation}
To see this we first note that, by Lemma \ref{weaksolismildsol}
below, $u$ can be written in the form
\[
\begin{split}
 u(x,t)=&\int_{\mathbb
R}G(x-y,t)u_0(y)dy  \\ &+\int_0^t\int_{\mathbb
R}G(x-y,t-s)\bar{b}(y)u(y,s)(1-u(y,s))\,dyds.
\end{split}
\]
For any $x\in\mathbb R,$ the first integral
\[
\ \int_{\mathbb R}G(x-y,t)u_0(y)dy \to u_0(x),\hbox{ as }t\to 0.
\]
By Lemma \ref{lem:A} the second integral can be estimated as
follows
\[
\begin{split}
 &\Big|\int_0^t\int_{\mathbb R}G(x-y,t-s)\bar{b}(y)u(y,s)(1-u(y,s))\,dyds\Big|\\
 \leq&\int_0^t\frac{D}{\sqrt{t-s}}ds,
\end{split}
\]where $D$ is a constant depending on $\|u(\cdot,t)\|_{L^\infty(\mathbb
R)}.$ Since $\|u(\cdot,t)\|_{L^\infty(\mathbb R)}$ is bounded, we
have
\[
\ \int_0^t\int_{\mathbb
R}G(x-y,t-s)\bar{b}(y)u(y,s)(1-u(y,s))\,dyds\to 0,\hbox{ as }t\to
0.
\] Consequently \eqref{initial condition} holds.

Next, let $u_0,\tilde{u}_0\in C(\mathbb R)\cap L^\infty(\mathbb
R)$ be arbitrary and let $u,\tilde{u}$ be the corresponding weak
solutions of \eqref{equationinsection3}, the existence of which
has been proven above. Then $w:=u-\tilde{u}$ satisfies
\begin{equation}\label{w-eq}
\left\{\begin{array}{ll} w_t=w_{xx}+m(x,t)w
\quad\ \ &(x\in\mathbb R,\;t>0),\vspace{5pt}\\
w(x,0)=w_0(x)\quad\ \ &(x\in\mathbb R),
\end{array}\right.
\end{equation}
where $m:=\bar{b}(x)(1-u-\tilde{u})$ is a measure-valued function
of $t$. By Lemma \ref{weaksolismildsol} below, we can express $w$
as
\begin{equation}\label{integral}
\begin{split}
&w(x,t)=\int_{\mathbb R} G(x-y,t)w_0(y)dy\\
&\hspace{65pt}+\int_0^t\int_{\mathbb
R}G(x-y,t-s)m(y,s)w(y,s)\,dy\,ds.
\end{split}
\end{equation}
Define
\[
\rho(t)=\Vert w(\cdot,t)\Vert_{L^\infty(\mathbb R)}.
\]
Then, since
\[
\int_{\mathbb R}G(x-y,t-s)m(y,s)w(y,s)\,dy\leq \Vert
m(\cdot,s)\Vert \max_{y\in\mathbb R}\,G(x-y,t-s)|w(y,s)|,
\]
we have
\begin{equation}\label{ineq}
\rho(t)\leq \rho(0)+M\int_0^t \frac{\rho(s)\,ds}{\sqrt{4\pi
(t-s)}},
\end{equation}
where $M$ is a constant such that $\Vert m(\cdot,t)\Vert\leq M$
for $t\geq 0$. By Lemma 7.7 of Alfaro, Hilhorst, Matano \cite{s3},
it follows that
\begin{equation}\label{est-rho}
\rho(t)\leq e^{M^2t/4}\Big(1+\frac{M}{\sqrt{4\pi}} \int_0^t
\frac{e^{-M^2s/4}}{\sqrt{s}}\,ds\Big)\rho(0)
=O\Big(e^{M^2t/4}\big(1+\sqrt{t}\,\big)\rho(0)\Big).
\end{equation}
Consequently
\[
\ \|u(\cdot,t)-\tilde{u}(\cdot,t)\|_{L^\infty(\mathbb R)}\leq
e^{M^2t/4}\Big(1+\frac{M}{\sqrt{4\pi}} \int_0^t
\frac{e^{-M^2s/4}}{\sqrt{s}}\,ds\Big)\|u_0-\tilde{u}_0\|_{L^\infty(\mathbb
R)}.
\]
This proves the continuous dependence on the initial data and the
uniqueness of the mild solution.

%Next we let $u$ and $\tilde{u}$ be two solutions of equation
%\eqref{equationinsection3} with initial data $u_0(x)$. Let
%$w=u-\tilde{u}.$ Then $m(x,t)$ would be
%$\bar{b}(x)(1-u-\tilde{u}).$ The associated $\rho(0)=0$, so we get
%$u\equiv\tilde{u}$ as long as they exist. {}From the estimate, it
%is not difficult to get the continuous dependency on the initial
%data.\vskip8pt
\end{proof}

\begin{lem}\label{weaksolismildsol} The function $u(x,t)$ in the proof of Theorem \ref{wlposedness} is also a mild solution. In other words,
$u(x,t)$ can be written as
\begin{equation}\label{Mildsolution}
\begin{split}
 u(x,t)=&\int_{\mathbb
R}G(x-y,t)u_0(y)dy  \\ &+\int_0^t\int_{\mathbb
R}G(x-y,t-s)\bar{b}(y)u(y,s)\big(1-u(y,s)\big)\,dyds,
\end{split}
\end{equation}
where the second integral on the right-hand side is understood as
in \eqref{baslinearfunctional}.
\end{lem}
\begin{proof}
Denote $u(y,s)(1-u(y,s))$ by $f(u(y,s)).$  Since $u_n$ is a
classical solution of equation \eqref{equationforun}, it is also a
mild solution. Hence
\[
\ u_n(x,t)=\int_{\mathbb R}G(x-y,t)u_0(y)+\int_0^t\int_{\mathbb
R}G(x-y,t-s)b_n(y)f(u_n(y,s))\,dyds.
\]
To prove this lemma, it is sufficient to prove
\[
\begin{split}
\ \int_0^t\int_{\mathbb R}&G(x-y,t-s)\bar{b}(y)f(u(y,s))\,dyds\\
&=\lim_{n\to\infty}\int_0^t\int_{\mathbb
R}G(x-y,t-s)b_n(y)f(u_n(y,s))\,dyds.
\end{split}
\]
We consider
\[
\ \int_0^t\int_{\mathbb
R}G(x-y,t-s)\big(b_n(y)f(u_n(y,s))-\bar{b}(y)f(u(y,s))\big)\,dyds,
\]
which we denote by $Z_n(x,t).$ First we have

\begin{equation}\label{Toprovemildsol}
\begin{split}
|Z_n(x,t)|&\leq\Big|\int_0^t\int_{\mathbb R}G(x-y,t-s)\big(b_n(y)-\bar{b}(y)\big)f(u(y,s))\,dyds\Big|\\
&+\Big|\int_0^t\int_{\mathbb
R}G(x-y,t-s)b_n(y)\big(f(u(y,s))-f(u_n(y,s))\big)\,dyds\Big|.
\end{split}
\end{equation}

Since $u(x,t)$ is bounded and $\{b_n-\bar{b}\}$ are uniformly
bounded linear functionals on any interval $[kL,(k+1)L]$ where
$k\in\mathbb Z,$ by Lemma \ref{lem:A} we have
\[
\begin{split}
&\Big|\int_{\mathbb
R}G(x-y,t-s)\big(b_n(y)-\bar{b}(y)\big)f(u(y,s))\,dy\Big|\\
&\hspace{30pt}\leq \sum_{k=-\infty}^{\infty} 2\alpha L \max_{y\in[0,L]}G(x-y-kL,t-s)\|f(u)\|_{L^\infty(\mathbb R\times[0,t])}\\
&\hspace{30pt}\leq 2\alpha L\|f(u)\|_{L^\infty(\mathbb
R\times[0,t])}\frac{C_1}{\sqrt{t-s}}
\end{split}
\] for some constant $C_1>0.$  Since
\[
\ \int_0^t\frac{C_1}{\sqrt{t-s}}ds<+\infty,
\] by Lebesgue convergence theorem,
\[
\ \lim_{n\to\infty}\int_0^t\int_{\mathbb
R}G(x-y,t-s)\big(b_n(y)-\bar{b}(y)\big)(f(u(y,s)))\,dyds=0.
\]
Similarly,
\[
\ \int_{\mathbb
R}G(x-y,t-s)b_n(y)\big(f(u(y,s))-f(u_n(y,s))\big)dy\leq\frac{C_2}{\sqrt{t-s}}
\] for some constant $C_2>0.$ Again, by Lebesgue convergence
theorem,
\[
\ \lim_{n\to\infty}\int_0^t\int_{\mathbb
R}G(x-y,t-s)b_n(y)\big(f(u(y,s))-f(u_n(y,s))\big)\,dyds=0.
\]
Hence
\[
\begin{split}
 u(x,t)=&\int_{\mathbb
R}G(x-y,t)u_0(y)dy  \\ &+\int_0^t\int_{\mathbb
R}G(x-y,t-s)\bar{b}(y)u(y,s)(1-u(y,s))\,dyds.
\end{split}
\]
%Moreover, we have $u(x,t)$ is bounded. So we obtain
%\[
%\ \lim_{t\searrow 0}u(x,t)=u_0(x).
%\]
\end{proof}

{}From Lemma \ref{EquicontinuousForX} in Section 6 and Theorem
\ref{wlposedness}, the following proposition can be obtained
easily.
\begin{pro}\label{compact}
Let $u(x,t,u_0)$ be the mild solution of
\eqref{equationinsection3} with initial data $u_0$. Then for any
$t_0>0$ and $M>0$, the family of functions
$\{u(x,t_0,u_0)\}_{\|u_0\|_{L^\infty(\mathbb R)}\leq M}$ is
uniformly equicontinuous in $x$.
\end{pro}
 Consider the construction and the
uniqueness of mild solution. The following comparison principle
holds.
\begin{pro}[Comparison principle] Let $u(x,t,u_0)$ be as in Proposition \ref{compact}. Then $u_0\leq v_0$ implies
\[
\ u(x,t,u_0)\leq u(x,t,v_0)\quad\hbox{ for }x\in\mathbb R,\,t\geq
0.
\]
\end{pro}

\fbox{Change begins}

\begin{Rema}\label{RemarkForMildSolution}
If $u(x,t)$ is a mild solution of \eqref{equationinsection3}, then
for any constant $\tau\geq 0,$ $u(x,x+\tau)$ is a mild solution of
\eqref{equationinsection3} with initial data $u(x,\tau).$ This is
obvious if $\bar{b}$ is a smooth function, since a mild solution
is a classical solution. In the general case where
$\bar{b}\in\overline{\Lambda}(\alpha).$ We can take a sequence of
smooth $b_n$ with $b_n\to\bar{b}$ (in the sense of ??) and use the
approximation argument found in Lemma \ref{weaksolismildsol}.
\end{Rema}

\fbox{Change ends}

\section{The linear eigenvalue problem}\label{Lineareigenvalueproblem}

\subsection{Basic estimates}
We recall that $\mu(\lambda,b)$ denotes the principal eigenvalue
of the problem
\[
-\psi''+2\lambda\psi'-b\psi=\mu(\lambda,b)\psi,\,\,\,\,\psi(x+L)\equiv\psi(x)
\](see Definition
\ref{definitionofprincipaleigenvalue}). In this subsection, we
estimate $\mu(\lambda,b)$ both from the above and below. Here we
introduce the following notation. Let
\[
\ \Hp :=\{ \varphi\in
H_{loc}^1(\mathbb{R})\,|\,\varphi(x)\equiv\varphi(x+L)\}
\] with the norm
\[
\
\|\varphi\|_{\Hp}=\big(\int_{[0,L)}(\varphi'^2+\varphi^2)dx\big)^{1/2}.
\]  Let
\[ E_L=\{\,\psi\in \Hp\,|\,\psi(x)>0\,\}. \]

 Note that the following embedding is compact:
\[
\ \Hp \hookrightarrow C(\mathbb{R})\cap L^{\infty}(\mathbb{R}).
\]

\begin{lem}\label{IntegralInequality}
For any $b \in \Lambda(\alpha)$ and any $\psi(x) \in E_L$, we have
\[ \ \int_{[0,L)} \psi'^2(x)dx-\int_{[0,L)} b(x)\psi^2(x)dx\geq
-(\alpha+\alpha^2L^2)\int_{[0,L)}\psi^2(x)dx.
\]
\end{lem}

\begin{proof}
First, for any $x_1, x_2 \in [0,L]$, we have \[ \
\psi^2(x_2)-\psi^2(x_1)= \int_{x_1}^{x_2}2\psi(x)\psi'(x)dx.
\]
Hence, for any positive number $k>0$, we have\[ \
\psi^2(x_2)-\psi^2(x_1)\leq
\frac{1}{k}\int_{[0,L)}\psi'^2(x)dx+k\int_{[0,L)}\psi^2(x)dx. \]
Multiplying the above inequality by $b(x_2)$ and integrating it by
$(x_1,x_2)\in[0,L]\times[0,L]$, we get
\[\begin{split}
&\ L\int_{[0,L)}b(x_2)\psi^2(x_2)dx_2-\alpha L
\int_{[0,L)}\psi^2(x_1)dx_1\\ &\hspace{50pt}\leq \alpha
L^2\Big(\frac{1}{k}\int_{[0,L)}\psi'^2(x)dx+k\int_{[0,L)}\psi^2(x)dx\Big).
\end{split}\] This is equivalent to
\[
\begin{split}
 &L\int_{[0,L)}b(x)\psi^2(x)dx-\alpha L \int_{[0,L)}\psi^2(x)dx\\
 &\hspace{50pt}\leq
\alpha
L^2\Big(\frac{1}{k}\int_{[0,L)}\psi'^2(x)dx+k\int_{[0,L)}\psi^2(x)dx\Big).
\end{split}
\]
Letting $k=\alpha L$, we obtain
\[
\ \int_{[0,L)}\psi'^2(x)dx-\int_{[0,L)}b(x)\psi^2(x)dx\geq
-(\alpha+\alpha^2L^2)\int_{[0,L)}\psi^2(x)dx.
\]
\end{proof}
\begin{lem}\label{LamdaEqual0}
For $\lambda>0$, $b \in \Lambda(\alpha)$, it holds that
$\mu(\lambda,b)\geq \mu(0,b).$
\end{lem}
\begin{proof}
Let $\psi>0$ be the principal eigenfunctions of $-L_{\lambda,b}.$
Then
\[
 -\psi''+2\lambda \psi'-b(x)\psi=\mu(\lambda,b)\psi.
\]
Multiplying this by $\psi$ and integrating it from $0$ to $L$, we
have
\[
\int_{[0,L)}\psi'^2(x)dx-\int_{[0,L)}b(x)\psi^2(x)dx=\mu(\lambda,b)\int_{[0,L)}\psi^2(x)dx,
\] hence
\[
\
\mu(\lambda,b)=\frac{\int_{[0,L)}\psi'^2(x)dx-\int_{[0,L)}b(x)\psi^2(x)dx}{\int_{[0,L)}\psi^2(x)dx}.
\]
On the other hand, for $\mu(0,b)$, we have the variational formula
\begin{equation} \label{variational} \ \mu(0,b)=\min_{\phi \in
E_L}\frac{\int_{[0,L)}\phi'^2(x)dx-\int_{[0,L)}b(x)\phi^2(x)dx}{\int_{[0,L)}\phi^2(x)dx}.
\end{equation}
Consequently $\mu(\lambda,b)\geq \mu(0,b).$
\end{proof}
\begin{lem}\label{Boundednessofmu}
For any $b\in\Lambda(\alpha),$ it holds that
\begin{equation}\label{estimateformu}
-\alpha\geq \mu(\lambda,b)\geq -\alpha-\alpha^2L^2.
\end{equation}
\end{lem}
\begin{proof}
Dividing \eqref{identityofmulambdabbar} by $\psi(x)$ and
integrating it from $0$ to $L,$ we get
\begin{equation}\label{dividingbypsiandintegrate}
(\mu(\lambda,b) +\alpha)L +\int_{[0,L)}
\big(\frac{\psi'}{\psi}\big)^2dx=0. \end{equation} It implies that
$\mu\leq -\alpha.$ By Lemma \ref{IntegralInequality} and
\ref{LamdaEqual0}, \[ \mu(\lambda,b)\geq\mu(0,b)\geq
-\alpha-\alpha^2L^2.
\]
\end{proof}
\begin{lem}\label{LinearPsi}
There exists a constant $F>0$ such that for any $\lambda\geq 0,$
$b\in\Lambda(\alpha),$ the principal eigenfunction $\psi\in E_L$
 of the operator $-L_{\lambda,b}$ satisfies
\[ \frac{\max\psi}{\min\psi}\leq F. \]
\end{lem}

\begin{proof}
By equation \eqref{dividingbypsiandintegrate}, for any $x_1,x_2\in
[0,L),$
\[
\begin{split}
 \Big|\ln\frac{\psi(x_2)}{\psi(x_1)}\Big|&=|\ln\psi(x_2)-\ln\psi(x_1)|=\Big|\int_{x_1}^{x_2}\frac{\psi'}{\psi}dx\Big|\\
  &\leq\int_{[0,L)}\big|\frac{\psi'}{\psi}\big|dx\leq\sqrt{L}\;\sqrt{\int_{[0,L)}\big(\frac{\psi'}{\psi}\big)^2dx }\\
 &\leq L\sqrt{-(\mu(\lambda,b)+\alpha)}.
\end{split}
\]
Then for any $x_1,x_2\in[0,L],$ \[ \ e^{-L\sqrt{-(\mu(\lambda,b)
+\alpha})}\leq\frac{\psi(x_1)}{\psi(x_2)}\leq
e^{L\sqrt{-(\mu(\lambda,b) +\alpha})}.
\] This implies \[
\ \frac{\max\psi}{\min\psi} \leq e^{L\sqrt{-(\mu(\lambda,b)
+\alpha})}.
\]
Note that
\[
\ \sqrt{-(\mu(\lambda,b)+\alpha)}\leq\sqrt{\alpha^2L^2}=\alpha L.
\]
Therefore, by setting $F=e^{\alpha L^2},$ we obtain the desired
estimate.
\end{proof}
\subsection{Uniform bounds of $c^*(b)$ for smooth $b$}\label{boundednessSection}
In this subsection, we will prove Proposition \ref{boundedness}.
Before doing that, we recall the following pointwise max-min
formula for the principal eigenvalue.
\begin{pro}\label{maxminrepresentation} For $b\in \Lambda(\alpha),$
\[
\ \mu(\lambda,b)=\max_{\psi \in E_L\cap C^2(\mathbb{R})} \,\inf_{x
\in\mathbb
R}\frac{-\psi''(x)+2\lambda\psi'(x)-b(x)\psi(x)}{\psi(x)}.
\]
\end{pro} We omit the proof of the above proposition as it is easy. A more general version of the above proposition is fonud in [\,4\,;\,Proposition 5.7\,].
 The following lemma easily follows form the definition
$c_e^*(\bar{b})$ in Definition \ref{Defofcestar}.
\begin{pro}\label{weinbergereigenvalue}
For $b\in\Lambda(\alpha),$
\begin{equation}\label{PropositionForSpeedCestar}
\
c_e^*(b)=\min_{\lambda>0}\frac{-\mu(\lambda,b)+\lambda^2}{\lambda}.
\end{equation}
\end{pro}

\begin{proof}[Proof of Proposition \ref{boundedness}] By
Proposition \ref{cstarbequalscestarbforsmoothb}, it suffices to
show that
\[
\ 2\sqrt{\alpha}\leq c_e^*(b)\leq 2\sqrt{\alpha+\alpha^2L^2}.
\] This follows immediately from \eqref{estimateformu} and
\eqref{PropositionForSpeedCestar}. The proposition is proven.
\end{proof}

\subsection{Uniform bounds of $c_e^*(\bar{b})$ when $\bar{b}$ is a
measure}\label{uniformboundofcestarbbar}

In the previous subsection, we discussed the boundedness of
$c^*(b)$ for $b\in\Lambda(\alpha).$ In this subsection, we derive
the same bounds for $c_e^*(\bar{b})$ when
 $\bar{b}\in\overline{\Lambda}(\alpha).$ We start with the following proposition:

\fbox{Change begins}
\begin{pro}\label{convergenceofeigenvalue}
Let $b_n$ be a sequence in $\Lambda(\alpha)$ converging to some
$\bar{b}$ in the weak \,$*$ sense and let $\lambda_n\in\mathbb{R}$
be a sequence converging to some $\lambda\in\mathbb{R}.$ Then
there exists a constant $\beta$ such that
\[
\mu(\lambda_n,b_n)\to\beta\,\,\,\,\,\,\hbox{ as }n\to\infty.
\] Furthermore,
\[
\ \beta=\mu(\lambda,\bar{b}).
\]
\end{pro}

\begin{proof} Lemma \ref{Boundednessofmu} shows that $\{\mu(\lambda_n,b_n)\}$ is
uniformly bounded. There exist functions $\psi_n\in E_L$ with
$\|\psi_n\|_{L^{\infty}}=1$ satisfying
\[
\ -\psi_n''+2\lambda_n\psi_n'-b_n\psi_n=\mu(\lambda,b_n)\psi_n.
\]
By Lemma \ref{LinearPsi}, there exists a constant $F>0$ such that
${\max \psi_n}/{\min\psi_n}<F.$ \fbox{Change begins} Now we
multiply
\[ \
-\psi''_n+2\lambda_n\psi'_n-(\mu(\lambda,b_n)+b_n(x))\psi_n=0
\] with $\psi_n$ and integrate it from $0$ to $L.$ We get
\[
\
\int_{[0,L)}(\psi'_n)^2dx=\int_{[0,L)}(\mu(\lambda,b_n)+b_n(x))\psi_n^2\,dx.
\]
Hence
\[
\
\int_{[0,L)}(\psi'_n)^2dx+\int_{[0,L)}(\psi_n)^2dx=\int_{[0,L)}\big(1+\mu(\lambda,b_n)+b_n(x)\big)\psi_n^2\,dx.
\]
By the fact that $\|\psi_n\|_{L^\infty}=1$ and
\eqref{estimateformu}, we see from the equation above that
$\{\|\psi_n\|_{\Hp}\}$ is uniformly bounded. \fbox{Change ends}
Then there exists $\psi\in \Hp$ such that for some $\beta$ there
exists a subsequence $\{\psi_{n_k}\},$
\[
\ \psi_{n_k}\to\psi \,\,\,\,\hbox{   weakly in } \Hp \hbox{ and
strongly in } C(\mathbb{R})
\]and $\mu(\lambda_{n_k},b_{n_k})\to\beta$ as $n_k\to +\infty.$
Therefore
\[
\ -\psi''+2\lambda\psi'-\bar{b}\psi=\beta\psi \,\,\,\,\,\hbox{ in
the weak sense,}
\] So $\mu(\lambda,\bar{b})=\beta.$ We will prove the uniqueness of $\mu(\lambda,\bar{b})$ in Proposition
\ref{uniqueeigenvalue}.\fbox{Check} Therefore, it is not difficult
to see $\mu(\lambda_n,b_n)\to\mu(\lambda,\bar{b}).$

\end{proof}

\fbox{Change ends}

The following corollary follows immediately from Proposition
\ref{convergenceofeigenvalue}:
\begin{corol}
Let $b_n$ be a sequence in $\Lambda(\alpha)$ converging to
$\bar{b}\in\overline{\Lambda}(\alpha)$ in the weak \,$*$ sense.
Then $\mu(\lambda,b_n)\to\mu(\lambda,\bar{b})$ locally uniformly
in $\lambda\geq 0.$
\end{corol}
We also remark that the continuity of
$\lambda\mapsto\mu(\lambda,\bar{b})$ follows easily from
Proposition \ref{convergenceofeigenvalue}.
\begin{proof}[Proof of Proposition \ref{Convergenceofcstarbn}]
By Definition \ref{Defofcestar}, we have
\[
\begin{split}
 &c_e^*(b_n):=\inf\{c>0\,|\,\exists \lambda>0 \hbox{ such that }\mu(\lambda,b_n)=\lambda^2-\lambda c\}\\
 &c_e^*(\bar{b}):=\inf\{c>0 \,| \,\exists \lambda>0 \hbox{ such that }\mu(\lambda,\bar{b})=\lambda^2-\lambda c \}.
\end{split}
\]
Since $\mu(\lambda,b_n)\to\mu(\lambda,\bar{b})$ locally uniformly
in $\lambda$ and since $\mu(\lambda,b_n)$ is uniformly bounded by
Lemma \ref{Boundednessofmu}. We immediately obtain the conclusion.

%Since $\{c^*(b_n)\}$ is uniformly bounded, we may find a
%subsequence $\{c^*(b_{n_k})\}$ such that there exists a constant
%$\beta>0,$ $\lim_{n_k\to\infty}c^*(b_{n_k})=\beta.$ By Lemma
%\ref{generalizedpsiexists}, we may find $\lambda^*>0,$ $ \psi^*\in
%\Hp$ satisfying $\psi^*(x)>0,\,\psi^*(x)\equiv\psi^*(x+L)$ such
%that
%\[
%\
%-{\psi^*}''(x)+2\lambda^*{\psi^*}'-({\lambda^*}^2-\lambda^*\beta+\bar{b})\psi^*(x)=0.
%\] This means $c_e^*(\bar{b})\leq \beta.$%
%
%Now we need to show that $c_e^*(\bar{b})= \beta.$ We know that
%there exists a $\lambda>0,$ such that
%$\mu(\lambda,\bar{b})-(\lambda^2-\lambda c_e^*(\bar{b}))=0.$
%{}From this equation, we get
%\[
%\
%\mu(\lambda,b_{n_k})-\big(\lambda^2-\lambda(c_e^*(\bar{b})-\frac{\mu(\lambda,b_{n_k})-\mu(\lambda,\bar{b})}{\lambda})\big)=0.
%\]
%According to the definition of $c_e^*(b_{n_k}),$ we get
%$c_e^*(b_{n_k})\leq
%c_e^*(\bar{b})-\frac{\mu(\lambda,b_{n_k})-\mu(\lambda,\bar{b})}{\lambda}.$
%Let ${n_k}\to\infty.$ We have $\beta\leq c_e^*(\bar{b}).$ So we
%get $c_e^*(\bar{b})= \beta.$ {}From the uniqueness of
%$c_e^*(\bar{b}),$ we have $c_e^*(b_n)\to c_e^*(\bar{b}).$
\end{proof}

Combining Proposition \ref{boundedness} and
\ref{Convergenceofcstarbn}, and recalling that every
$\bar{b}\in\overline{\Lambda}(\alpha)$ can be expressed as a weak
\,$*$ limit of sequence in $\Lambda(\alpha)$ (see
\eqref{DefinitionOfLamdabarAlpha}), we obtain the following
proposition.
\begin{pro} For any $\bar{b}\in\overline{\Lambda}(\alpha),$
\[
\ 2\sqrt{\alpha}\leq c_e^*(\bar{b})\leq
2\sqrt{\alpha+\alpha^2L^2}.
\]
\end{pro}
\subsection{Maximizing $c_e^*(\bar{b})$}

In Subsection \ref{uniformboundofcestarbbar}, we have shown that
$c_e^*(\bar{b})$ is bounded. In this subsection, we consider the
variational problem
\[
 \underset{\bar{b}\in\overline{\Lambda}(\alpha)}{{\rm Maximize\quad}}c_e^*(\bar{b})
\] and show that the maximum is attained by $h(x)$ defined in \eqref{h}.
\begin{lem}\label{cestarhissupofcestarbbar}
\[
 c_e^*(h)=\max_{\bar{b}\in\overline{\Lambda}(\alpha)}c_e^*(\bar{b}).
\]
\end{lem}
To prove the above result, we need some auxiliary lemmas.
\begin{lem} If there exist $\lambda>0,$ $c>0$ and $\psi\in E_L$ such that
\begin{equation}\label{5}
\ -\psi''+2\lambda \psi'-(\lambda^2-\lambda c+h(x) )\psi=0,
\end{equation}
then $\psi$ is a piecewise $C^1$ function and $\psi(L/2)\geq
\psi(x)$ for $x\in \mathbb R.$
\end{lem}
\begin{proof}
First, we just integrate \eqref{5} from $0$ to $L$ to get \[ \
\int_{[0,L)}(\lambda^2-\lambda c+h(x) )\psi dx=0.
\]
This implies $\lambda^2-\lambda c<0.$

Since $h(x)=0$ for $x \in(-L/2,L/2),$
\[ \ -\psi''+2\lambda
\psi'-(\lambda^2-\lambda
c)\psi=0\,\,\;\;\,\,(-\frac{L}{2}<x<\frac{L}{2})
\] in the classical sense. Furthermore, $\psi$ is continuous up to $x=\pm
\frac{L}{2}$ since $E_L\subset C(\mathbb{R}).$ Consequently by the
classical maximum principle and the positivity of $\psi(x),$ along
with the negativity of $\lambda^2-\lambda c,$ the maximum of
$\psi$ must be attained at $x=\pm \frac{L}{2}$. Hence
\begin{equation} \label{psigetitsmaximumatL2}
\psi\big(\frac{L}{2}\big)>\psi(x),\quad \hbox{for
}x\in(-\frac{L}{2},\frac{L}{2}).\end{equation} {}From the
periodicity of $\psi(x),$ we get that \[ \
\psi\big(\frac{L}{2}\big)\geq \psi(x),\,\,\,x\in\mathbb R.
\]
Since $\psi$ can be expressed as
\[
\ \psi(x)=c_1e^{\nu_1x}+c_2e^{\nu_2x} \,\,\,\,\,\,\,\hbox{ in
}(-\frac{L}{2},\frac{L}{2})
\] for some constants $c_1,c_2,\nu_1,\nu_2,$
$\psi'(\frac{L}{2}-0)$ and $\psi'(\frac{L}{2}+0)$ both exist.
Consequently, $\psi$ is piecewise $C^1$ on $\mathbb{R}.$
\end{proof}

\begin{lem}\label{cehlargerthancb}
For any $b\in \Lambda(\alpha),$ $c_e^*(h)> c^*(b)$.
\end{lem}

\begin{proof}
In what follows $c$ will stand for $c_e^*(h)$ for simplicity.
Then, by the definition of $c_e^*(h),$ there exists $\psi\in E_L$
such that
\[ \ -\psi''+2\lambda \psi'-(\lambda^2-\lambda c+h(x) )\psi=0
\] in the weak sense for some $\lambda>0.$ We define the
convolution $b\ast\psi$ for $L-$periodic functions by
\[
\ [b*\psi](x):=\int_{[0,L)}\psi(x-y)b(y)dy.
\] Then one can easily see that $\tilde{\psi}:=b\ast\psi$ belongs to $E_L\cap C^2(\mathbb{R})$ since $b$ is $C^1$ and $\psi$ is piecewise $C^1.$
Consequently,

\[
\begin{split}
\ 0&=b\ast\big(-\psi''+2\lambda\psi'-(\lambda^2-\lambda c+b(x))\psi\big) \\
&=-b\ast\psi''+2\lambda
b\ast\psi'-b\ast((\lambda^2-\lambda c+h(x) )\psi) \\
&=-\tilde{\psi}''+2\lambda \tilde{\psi}'-(\lambda^2-\lambda
c+b(x) )\tilde{\psi}-(b\ast(h\psi)-b(b\ast\psi))\\
 &=-\tilde{\psi}''+2\lambda
\tilde{\psi}'-(\lambda^2-\lambda c+b(x) )\tilde{\psi}-\big(\alpha
L\psi(\frac{L}{2})b-b(b\ast \psi)\big).
\end{split}
\]
Note that
\[
\begin{split}
 \alpha L\psi(\frac{L}{2})b(x)-b(x)[b\ast \psi](x) =&b(x)(\alpha
L\psi(\frac{L}{2})-[b\ast\psi](x))\\
=&b(x)\big([b\ast\psi(\frac{L}{2})](x)-[b\ast\psi](x)\big)\\=&b(x)\big[b\ast\big(\psi(\frac{L}{2})-\psi(x)\big)\big](x)\\
 >&0
\end{split}
\] by \eqref{psigetitsmaximumatL2}. Furthermore, since the
left-hand side of the above inequality is continuous and periodic in $x,$ we
have $\alpha L\psi(\frac{L}{2})b(x)-b(x)b\ast
\psi\geq\lambda\sigma$ for some $\sigma>0.$ It follows that
\[
\
-\tilde{\psi}''+2\lambda \tilde{\psi}'-(\lambda^2-\lambda
(c-\sigma)+b(x))\tilde{\psi}>0,
 \] hence, by Proposition \ref{maxminrepresentation},
\[
\
\lambda^2-\lambda(c-\sigma)\leq\frac{-\tilde{\psi}''+2\lambda\tilde{\psi}'-b(x)\tilde{\psi}}{\tilde{\psi}}\leq
\mu(\lambda,b).
\] Combining this and Proposition \ref{weinbergereigenvalue}, we
obtain
\[
\ c_e^*(b)\leq c-\sigma < c:=c_e^*(h).
\] Moreover, by Proposition
\ref{cstarbequalscestarbforsmoothb},
\[
\ c^*(b)<c_e^*(h)
\] for any $b\in\Lambda(\alpha).$
\end{proof}

\begin{proof}[Proof of Lemma \ref{cestarhissupofcestarbbar}]
The proof is immediate from Lemma \ref{cehlargerthancb} and
Proposition \ref{Convergenceofcstarbn}.
\end{proof}

\section{Pulsating travelling waves}\label{pulsating travelling waves}

In this section we will prove our main results: Theorems
\ref{cstarbequalcestarb}, \ref{cestarhlargerthancestarb},
\ref{Themforspreadingspeed}.

In Subsection 5.1, we show that the minimal speeds $c^*(\bar{b})$
and $\tilde{c}^*(\bar{b})$ exist and that they coincide with the
spreading speeds $c^{**}(\bar{b})$ and $\tilde{c}^{**}(\bar{b}),$
respectively. In Subsection 5.2, we prove
$c^*(\bar{b})=c_e^*(\bar{b})$ and
$\tilde{c}^*(\bar{b})=\tilde{c}_e^*(\bar{b})$ thereby completing
the proof of Theorem \ref{cstarbequalcestarb}. In Subsection 5.3,
we complete the proof of Theorems \ref{cestarhlargerthancestarb}
and \ref{Themforspreadingspeed}. Without loss of generality, in
this section, let $L=1.$

\subsection{The existence of the minimal speed and spreading speed}
Our strategy here is to use the general results of Weinberger
\cite{s8} and Liang and Zhao \cite{s7}. For this purpose, we need
to consider the solution semiflow of \eqref{Mainequation} on the
space $C(\mathbb{R};[0,1])$  with respect to the local uniform
topology, where the convergence $u_n\rightarrow u$ means that
$u_n(x)\rightarrow u(x)$ uniformly on any bounded interval. First,
we consider the linear space $BC(\mathbb{R})=\{u\,|\,u\in
C(\mathbb R)\cap L^\infty(\mathbb R)\}$ which contains
$C(\mathbb{R};[0,1])$. Equip $BC(\mathbb{R})$ with the topology of
locally uniform convergence. It is easy to show that this linear
topological space $BC(\mathbb{R})$ has the following properties.

\begin{pro}\label{equivalent norm} For each $m\in\mathbb Z$ and $n\in\mathbb
N,$ define $$\|u\|^n_m=\sum\limits_{i\in
\mathbb{Z}}\frac{\max\limits_{z\in[m+i,m+i+n]}|u(z)|}{2^{|i|}}.$$
Then each $\|\cdot\|^n_m$ is a norm on $BC(\mathbb{R})$ and for
any $M>0$, $\|\cdot\|^n_m$ defines a topology equivalent to the
local uniform topology on $C(\mathbb{R};[-M,M])$.

Furthermore, for any $u\in BC(\mathbb R),$

\fbox{Wrong begins}$$\|u\|^1_m\leq \|u\|^{n}_m \leq n\|u\|^{1}_m
,\,\,\,\,\,\forall m\in \mathbb{Z}, n\in \mathbb{N}$$ and
$$2^{-|m|}\|u\|^n_0\leq \|u\|^{n}_m \leq
2^{|m|}\|u\|^{n}_0,\,\,\,\forall m\in \mathbb{Z},n\in
\mathbb{N}.$$\fbox{Wrong ends}

\fbox{Modification begins}
$$
\|u\|_{m+j}^1\leq\|u\|_m^n\leq\sum_{i=1}^n\|u\|_{m+i}^1
\,\,\,\forall \,m\in \mathbb{Z},n\in \mathbb{N},j=0,\ldots,n-1\,$$
and
$$2^{-|m|}\|u\|^n_0\leq \|u\|^{n}_m \leq
2^{|m|}\|u\|^{n}_0,\,\,\,\forall\, m\in \mathbb{Z},n\in
\mathbb{N}.$$

\fbox{Modification ends}
\end{pro}

\begin{pro}
$C(\mathbb{R};[0,1])$ is a bounded closed subset of
$BC(\mathbb{R})$. Moreover, $C(\mathbb{R};[0,1])$ is complete.
\end{pro}
In Section 3, we have shown that the mild solutions of
\eqref{equationinsection3} depend on the initial data continuously
in $L^\infty(\mathbb{R}).$  In the following proposition, we will
show that the continuous dependence also holds with respect to the
local uniform topology.
\begin{pro}\label{lcontinuous}The mild solutions of
\eqref{equationinsection3} depend on the initial data continuously
in the space $C(\mathbb{R};[0,1])$ . Precisely, for any
$\epsilon>0$ and $T>0$, there is some $\eta>0$ such that for any
two solutions $u,v$ of (\ref{equationinsection3}) with initial
data $u_0,$ $v_0$, if $\|u_0-v_0\|_0^1\leq \eta$, then
$\|u(\cdot,t)-v(\cdot,t)\|_0^1\leq \epsilon$ for $t\in [0,T]$.

\end{pro}
\begin{proof}Let $w=u-v$ with $w_0=u_0-v_0$, then
\begin{equation}\label{integral}
\begin{split}
&w(x,t)=\int_{\mathbb R} G(x-y,t)w_0(y)dy\\
&\hspace{60pt}+\int_0^t\int_{\mathbb
R}G(y,t-s)\bar{b}(x-y)m(x-y,s)w(x-y,s)\,dy\,ds,
\end{split}
\end{equation}
where $m(x,s)=1-u(x,s)-v(x,s)$. By comparison principle, $0\leq
u_0(x),v_0(x)\leq 1,\,\forall x\in \mathbb{R}$ implies that $0\leq
u(x,t),v(x,t)\leq 1,\forall x\in \mathbb{R},t>0$, and then there
$|m(x,s)|\leq1$. Let
$$I_1(t,x)=\int_{\mathbb R} G(x-y,t)w_0(y)dy$$ and
$$I_2(t,x)=\int_0^t\int_{\mathbb
R}G(y,t-s)\bar{b}(x-y)(1-u(x-y,s)-v(x-y,s))w(x-y,s)\,dy\,ds.$$
First, we consider $I_1$. For any $\epsilon>0$ and $T>0$ there is
some $\eta>0$, such that if
$\|w_0\|_0^1=\|I_1(0,\cdot)\|_0^1<\eta$, then
$\|I_1(t,\cdot)\|_0^1<\epsilon$ for any $t\in [0, T]$. Then we
consider $I_2$. By Lemma \ref{lem:A},

\[ \begin{split}&\Big|\int_{\mathbb
R}G(y,t-s)\bar{b}(x-y)m(x-y,s)w(x-y,s)\,dy\,\Big| \\
\leq & M \alpha L\sum\limits_{i\in \mathbb{Z}}\max\limits_{y\in
[i,i+1]}G(y,t-s)\max\limits_{y\in [i,i+1]}|w(x-y,s)|\end{split}\]
where $M$ is a constant such that $\Vert m(\cdot,t)\Vert\leq M$
for $t\geq 0$..

Hence, we have\[\begin{split}\max\limits_{x\in
[k,k+1]}|I_2(t,x)|&\leq M\alpha L\int_0^t\max\limits_{x\in
[k,k+1]}\sum\limits_{i\in \mathbb{Z}}\max\limits_{y\in
[i,i+1]}G(y,t-s)\max\limits_{y\in [i,i+1]}|w(x-y,s)|ds\\ &\leq
M\alpha L\int_0^t\sum\limits_{i\in \mathbb{Z}}\max\limits_{y\in
[i,i+1]}G(y,t-s)\max\limits_{z\in [k-i-1,k-i+1]}|w(z,s)|ds.
\end{split}
\]
This implies that \[\begin{split}\|I_2(t,\cdot)\|_0^1&\leq M\alpha
L\int_0^t\sum\limits_{i\in \mathbb{Z}}\max\limits_{y\in
[i,i+1]}(G(y,t-s)\sum\limits_{k\in
\mathbb{Z}}2^{-|k|}\max\limits_{z\in [k-i-1,k-i+1]}|w(z,s)|)ds\\ &
=M\alpha L\int_0^t\sum\limits_{i\in \mathbb{Z}}\max\limits_{y\in
[i,i+1]}G(y,t-s)\|w(\cdot,s)\|_{-i-1}^2ds\\ & \leq M\alpha
L\int_0^t\sum\limits_{i\in \mathbb{Z}}\max\limits_{y\in
[i,i+1]}G(y,t-s)2^{|i+2|}\|w(\cdot,s)\|_{0}^1\\ & \leq
\int_0^t\frac{C'}{\sqrt{t-s}}\big\|w(\cdot,s)\big\|_{0}^1ds
\end{split}
\]
for some positive constant $C'.$ Then we have
$$\|w(\cdot,t)\|_{0}^1\leq \epsilon
+\int_0^t\frac{C'}{\sqrt{t-s}}\big\|w(\cdot,s)\big\|_{0}^1ds$$
provided $\|w_0\|_{0}^1\leq \eta$. It follows that
\[
\|w(\cdot,t)\|_{0}^1\leq e^{M^2t/4}\Big(1+\frac{M}{\sqrt{4\pi}}
\int_0^t \frac{e^{-M^2s/4}}{\sqrt{s}}\,ds\Big)\epsilon
=O\Big(e^{M^2t/4}\big(1+\sqrt{t}\,\big)\epsilon\Big)
\]
from Lemma 7.7 of Alfaro, Hilhorst, Matano \cite{s3}. The proof is
completed.
\end{proof}
We have shown in Section 3 that for any
initial function $u_0\in C(\mathbb{R};[0,1])$, the mild solution
$u(x,t,u_0)$ of \eqref{Mainequation} exists for any $t>0$ and that
$u(\cdot,t,u_0)\in C(\mathbb{R};[0,1])$. Now define an operator
$Q:C(\mathbb{R};[0,1])\times \mathbb R^+\rightarrow
C(\mathbb{R};[0,1])$ by
\[
\ Q_t(u_0)(x)=Q(u_0,t)(x)=u(x,t,u_0),
\]
where $u(x,t,u_0)$ is the solution of \eqref{Mainequation} with
initial data $u_0(x).$

In the following proposition, we will show that $Q$ is a semiflow
on the space $C(\mathbb{R};[0,1])$ with respect to the local
uniform topology, in other words, with respect to the norm
$\|\cdot\|_0^1$.
\begin{pro}\label{QIsSemiFlow}
$Q$ is a semiflow generated by the solution of equation
\eqref{Mainequation} in $C(\mathbb{R};[0,1])$ (with respect to
local uniform topology) in the following sense: \begin{enumerate}
\item[{\rm(a)}] $Q_0(u_0)=u_0$, \item[{\rm(b)}]
$Q_{t_1+t_2}=Q_{t_1}\circ Q_{t_2}$, \item[{\rm(c)}] $Q$ is
continuous in $(u_0,t)$.
\end{enumerate}
Moreover, for any $t>0,$ $\{Q_t(u_0)\,|\,u_0\in
C(\mathbb{R};[0,1])\}$ is precompact in $C(\mathbb{R};[0,1])$.
\end{pro}
\begin{proof}
The properties (a),(b) are obvious. To prove property (c), we only
need to show that \begin{enumerate} \item [(1)]For any given
$u_0$, $Q(u_0,t)$ is continuous in $t$, \item[(2)] For any given
$T>0$, the family of maps $Q_t: C(\mathbb R;[0,1])\to C(\mathbb
R;[0,1]),$ $0\leq t\leq T,$ is uniformly equicontinuous.
\end{enumerate}
(1) comes from Lemma \ref{weaksolismildsol}. In Proposition
\ref{lcontinuous}, we have shown that property (2) holds. This
completes the proof that $Q$ is a semiflow.

The compactness of $\{Q_t(u_0)\,|\,u_0\in C(\mathbb{R};[0,1])\}$
is equivalent to the uniform equicontinuity of the function
family. So it can be obtained from Proposition \ref{compact}
directly. The proof is complete.

\end{proof}

In the following lemma, we prove that the semiflow $Q$ is
monostable.

\begin{lem}
For continuous periodic initial data $1\geq u_0\geq 0,$
$u_0\not\equiv 0$, we have $u(x,t,u_0)\to 1,$ as $t\to\infty.$
\end{lem}

\begin{proof}
We consider interval $[0,L].$ We have
\[
\ u_t=u_{xx}+\bar{b}(x)u(1-u)
\]
with initial data $u(x,0)=u_0(x).$ Let $v$ be the solution of
\[
\ v_t=v_{xx}
\]
with initial data $v(x,0)=u_0(x).$ {}From the classical theory of
the heat equation,  $v(x,t)>0$ for any $t>0,$ $x\in[0,L].$

{}From the weak comparison principle, we have $1\geq
u(x,t,u_0)\geq 0.$ Moreover $\bar{b}\geq 0.$ Hence
\[
\ (u-v)_t\geq (u-v)_{xx},
\]
which means $u(x,t)\geq v(x,t)>0.$ For a given $t_1>0$, we may
find a constant $a>0$ such that $u(x,t_1)>a.$ This can be done
because $u(x,t_1)$ is periodic with respect to $x.$

Let $\tilde{v}(x,t)$ be the solution of
\[
\ \tilde{v}_t=\tilde{v}_{xx}+\bar{b}(x)\tilde{v}(1-\tilde{v})
\] with initial data $\tilde{v}(x,0)=a.$ Obviously \[
\ u(x,t+t_1,u_0)\geq \tilde{v}(x,t).
\]
Since $a>0$ is a sub-solution of equation
$\tilde{v}_t=\tilde{v}_{xx}+\bar{b}(x)\tilde{v}(1-\tilde{v})$,
from the comparison principle, we can get that $\tilde{v}(x,t)$ is
increasing in $t.$

Set
\[ \ P^+(x)=\lim_{t_n\to\infty}\tilde{v}(x,t_n).
\]
Then \[ \ P^+_{xx}+\bar{b}(x)P^+(1-P^+)=0
\] in the weak sense.
{}From $P^+_{xx}\leq 0$ and $P^+(x)$ is periodic and $P^+(x)\geq
a>0$, we have $P^+\equiv 1.$ Since $u(x,t+t_1,u_0)\geq
\tilde{v}(x,t)$, we have $u(x,t,u_0)\to 1,\,t\to\infty.$
\end{proof}
Summarizing, for any $t>0$, $Q_t$ has the following properties:
\begin{enumerate}
\item[{\rm(i)}] $Q_t$ is order preserving in the sense that if
$u_0,v_0\in C(\mathbb{R};[0,1])$ and $u_0(x)\leq v_0(x)$ on
$\mathbb{R}$, $Q_t(u_0)(x)\leq Q_t(v_0)(x)$ on $\mathbb{R}$.

\item[{\rm(ii)}] $Q_t(T_L(u_0))=T_L(Q_t(u_0))$,  where $T_L$ is a
shift operator with $T_L(u)(x)=u(x-L)$ .

\item[{\rm(iii)}] $Q_t(0)=0$ and $Q_t(1)=1$. For any $u_0\in
C(\mathbb{R};[0,1])$ with $u_0(x+L)=u_0$ on $\mathbb{R}$ and
$u\not \equiv 0$, $Q_t(u_0)\rightarrow 1$ in the space
$C(\mathbb{R};[0,1])$ with respect to the local uniform topology.

\item[{\rm(iv)}]  Given $T>0$, the family of maps $Q_t: C(\mathbb
R;[0,1]\to C(\mathbb R;[0,1])),\,0\leq t\leq T,$ is uniformly
equicontinuous with respect to the local uniform topology.

\item[{\rm(v)}] $Q_t(C(\mathbb{R};[0,1]))$ is precompact in
$C(\mathbb{R};[0,1])$ with respect to the local uniform topology.
\end{enumerate}

Thanks to these properties, the theorem holds:
\begin{them} \label{them5.6}
For any $\bar{b}\in\overline{\Lambda}(\alpha),$ the spreading
speed in the positive and negative directions $c^{**}(\bar{b})$
and $\tilde{c}^{**}(\bar{b})$ exist and they coincide with the
minimal travelling wave speed in the positive and negative
directions $c^*(\bar{b})$ and $\tilde{c}^*(\bar{b})$ respectively.
\end{them}
\begin{proof}
This theorem can be obtained from the results of
Weinberger\cite{s8} or a more abstract results of Liang, Zhao
\cite{s7}. In fact, for any $t>0,$ $Q_t$ satisfies (i)-(v) which
are Hypotheses 2.1 in \cite{s8}, and from Theorem 2.6 of \cite{s8}
we can get the existence of $c^*(\bar{b}),$
$\bar{b}\in\overline{\Lambda}(\alpha).$
\end{proof}

\subsection{Proof of $c^*(\bar{b})=c_e^*(\bar{b})$}

To prove $c^*(\bar{b})=c_e^*(\bar{b})$, we consider the linearized
equation  \[ \ u_t=u_{xx}+\bar{b}(x)u.
\] We define a linear space $\mathbb X$ by \[ \mathbb{X}:=\{\phi=e^{\xi_1x}\phi_1+e^{\xi_2x}\phi_2\,|\,
\xi_1,\xi_2\in \mathbb{R},\,\phi_1,\phi_2\in BC(\mathbb R)\}\] and
the subset $$\mathbb{X}_M^\xi:=\{\phi\in \mathbb{X}:|\phi(x)|\leq
Me^{\xi|x|}\}$$ for any $M,\xi>0$. Again we equip $\mathbb{X}$
with the local uniform topology too. Similar to the proof of
Theorem \ref{wlposedness} and Proposition \ref{lcontinuous}, we
can obtain
\begin{lem}
For any $\phi\in \mathbb{X}$, the mild solution $u(x,t)$ of the
equation
\[ \ u_t=u_{xx}+\bar{b}(x)u
\] with initial data $u(x,0)=\phi(x)$ exists for all $t>0$ and is
unique. The mild solution is a weak solution and for any $t>0$,
$u(\cdot,t,\phi)\in \mathbb{X}$. Moreover, the mild solutions
depend on the initial data continuously with respect the local
uniformly topology on $\mathbb{X}_M^\xi$ for any $\xi,M>0$.
\end{lem}
Define \[ \Phi_t(\phi)(x)=u(x,t,\phi), \,\,\,\forall
\phi\in\mathbb{X},\]
where $u(x,t,u_0)$ is the solution of \[ \
u_t=u_{xx}+\bar{b}(x)u
\] with initial data $u(x,0)=u_0(x).$
Similarly to the proof of Lemma \ref{Equicontinuity} and
Proposition \ref{QIsSemiFlow}, we can prove the following lemma:
\begin{lem} For any $t\geq 0$,
$\Phi_t$ is continuous on $\mathbb X$ with respect to the local
uniform topology. \end{lem}

Next, for any $M>0$, let $A^0_M=C(\mathbb{R};[M,M])$ and
$A^\xi_M=\{e^{\xi x}u\,|\,u\in A^0_M\}$ for $\xi\in \mathbb{R}$.
Then we have

\begin{lem}
$\Phi_t(A^\xi_M)$ is precompact in $\mathbb{X}$ with respect to
the local uniform topology.
\end{lem}

% For any $M>0$, let $A_M=\{u\in BC(\mathbb R),
%|u(x)|\leq M \, {\rm on}\, \mathbb{R}\}$. Then
%\end{lem}
%Furthermore,

In the following lemma, we show that $\Phi_t$ is {\textbf strongly
order-preserving}.
\begin{lem}
For any $u_0\in BC(\mathbb R)$ with $u_0(x)\geq 0,$ $\forall x\in
\mathbb{R},$ $u\not\equiv 0$ and any $t>0$, we have
$\Phi_t(u_0)(x)>0$ for $x\in\mathbb R.$
\end{lem}

\begin{proof}
Let $u(x,t,u_0)$ be the solution of
\[
\ u_t=u_{xx}+\bar{b}(x)u,
\]
with initial data $u_0,$ and let $v(x,t)$ be the solution of
\[
\ v_t=v_{xx}
\] with initial data $u_0(x).$

Since $\bar{b}(x)\geq 0$, we see that $v(x,t)$ is a sub-solution
of $u_t=u_{xx}+\bar{b}(x)u$ with initial data $u_0(x).$ Then we
get
\[ \ u(x,t)\geq v(x,t).
\]
{}From the classical theory of the heat equation, $v(x,t)>0$ for
$t>0.$ So we have $\Phi_t(u_0)=u(x,t,u_0)\geq v(x,t)>0.$
\end{proof}

Let $\Gamma=\{a(x)\in C(\mathbb{R})\cap L^\infty(\mathbb{R}),
a(x)=a(x+L)\}$. We equip $\Gamma$ with the local uniform topology,
which is also equivalent to the $L^\infty$ topology on $\Gamma$.
For any $\xi\geq 0$, define a linear operator $\mathcal{L}_t^\xi$
on $\Gamma$ by
\[
 \mathcal{L}_t^\xi(a)=e^{\xi x} \Phi_t(e^{-\xi x}a).
\]
{}From the definition and the properties of $\Phi_t$, we have

\begin{lem}
For any $t>0$ and $\xi\geq 0$, $\mathcal{L}_t^\xi:\Gamma
\rightarrow \Gamma$ is bounded, compact and strongly positive.

\end{lem}
\begin{pro}\label{uniqueeigenvalue}
$\psi$ is a principal eigenfunction of $-L_{\lambda,\bar{b}}$ if
and only if $\psi$ is a principal eigenfunction of
$\mathcal{L}_t^\lambda$. $\mu(\lambda,\bar{b})$ is a principal
eigenvalue of $L_{\lambda,\bar{b}}$ if and only if
$\exp(-\mu(\lambda,\bar{b})+\lambda^2)$ is a principal eigenvalue
of $\mathcal{L}_t^\lambda$. Consequently the principal eigenvalue
of $L_{\lambda,\bar{b}}$ is unique.
\end{pro}
\begin{proof}
Since $\psi$ is a principal eigenfunction of
$L_{\lambda,\bar{b}},$
\[ \
\psi''-2\lambda
\psi'+\bar{b}(x)\psi=-\mu(\lambda,\bar{b})\psi\,\,\,\,\,\,\,
\hbox{ in the weak sense}.
\]
Let$\phi=e^{-\lambda x}\psi$. Then the above formula is equivalent
to
\[ \
\phi''+\bar{b}(x)\phi=(-\mu(\lambda,\bar{b})+\lambda^2)\phi
\,\,\,\,\,\,\,\hbox{ in the weak sense.}
\]
\fbox{Change begins} It is easy to show that
$\exp((\mu(\lambda,\bar{b})+\lambda^2)t)\psi$ is the mild solution
of $u_t=u_{xx}+\bar{b}u$ with $u_0(x)=\psi(x).$ \fbox{Change ends}
Hence,
\[\Phi_t(\phi)=\exp((-\mu(\lambda,\bar{b})+\lambda^2)t)\phi .\]
Finally, we have \[\mathcal{L}_t^\lambda(\psi)=
\exp((-\mu(\lambda,\bar{b})+\lambda^2)t)\psi. \] Since
$\mathcal{L}_t^\lambda$ is a strongly positive compact operator,
by the Krein-Rutman theory, its principal eigenvalue is unique,
hence that of $L_{\lambda,\bar{b}}.$ This completes our proof.
\end{proof}
\begin{Rema} Combining Proposition \ref{uniqueeigenvalue} and Lemma \ref{convergenceofeigenvalue}, we know that Proposition
\ref{principaleigenvalueisunique} holds.
\end{Rema}
\begin{pro}
The principal eigenvalue of $-L_{0,\bar{b}}$ is negative, and then
the principal eigenvalue of $\mathcal{L}_t^0$ is lager than $1$
for any $t>0$.
\end{pro}

\begin{proof}
Integrating from $0$ to $L$, we get \[ \
\mu(0,\bar{b})=-\frac{\int_{[0,L)}\bar{b}(x)u(x)dx}{\int_0^Lu(x)dx}<0.
\]

Furthermore, the principal eigenvalue of $\mathcal{L}_t^0$ is
$e^{-\mu(0,\bar{b})t}>1$. This completes the proof.
\end{proof}

Summarizing, for any $t>0$, $\Phi_t$ has the following properties:
\begin{enumerate}
\item[{\rm(I)}] $\Phi_t$ is strongly order-preserving in the sense
that for any $u_0\in BC(\mathbb R)$ with $u_0(x)\geq 0, \forall
x\in \mathbb{R}, u\not\equiv 0$, $\Phi_t(u_0)(x)>0,$ for
$x\in\mathbb R.$

\item[{\rm(II)}] $\Phi_t(T_L(u_0))=T_L(\Phi_t(u_0))$,  where $T_L$
is a shift operator with $T_L(u)(x)=u(x-L)$ .

\item[{\rm(III)}]  For any $t>0$ and $\xi\geq 0$, the operator
$\mathcal{L}_t^\xi:\Gamma \rightarrow \Gamma$ is bounded, compact
and strongly positive. Moreover, the principal eigenvalue of
$\mathcal{L}_t^0$ is larger than 1.

\end{enumerate}
In the following lemma, we will show that $Q$ can be dominated by
$\Phi$ from above.
 \begin{lem} Let $Q_t$ and $\Phi_t$ be defined as above.
Then for any $u_0$ with $u_0(x)\geq 0,$ we have $Q_t(u_0)\leq
\Phi_t(u_0)$ for any $t>0, x\in \mathbb R.$
\end{lem}
\begin{proof}
We recall that $Q_t(u_0)(x)$ is the solution $u(x,t,u_0)$ of the
equation \[ \ u_t=u_{xx}+\bar{b}(x)u(1-u)
\] with initial data $u_0$ and $\Phi_t(u_0)(x)$ is the solution
$v(x,t,u_0)$ of the equation \[ \ v_t=v_{xx}+\bar{b}(x)v
\] with initial data $u_0.$

{}From the weak comparison principle, $0\leq u\leq 1.$ Let
$w=u-v.$ Then \[ \ w_t\leq w_{xx}+\bar{b}w.
\]
{}From the classical comparison principle of heat equation, we
have $w\leq 0.$ So $Q_t(u_0)\leq \Phi_t(u_0)$ for any $t>0, x\in
\mathbb R.$\end{proof}

On the other hand, we want to find some linear operator dominate
$Q$ from below.
\begin{lem}
For any $\epsilon$ with $0<\epsilon<1$, let $\Phi^\epsilon_t$ be
an operator such that
\[ \ \Phi^\epsilon_t (u_0)=u^{\epsilon}(x,t,u_0)
\] where $u^{\epsilon}(x,t,u_0)$ is the solution of \[
\ u^{\epsilon}_t=u^\epsilon_{xx}+(1-\epsilon)\bar{b}(x)u^\epsilon
\] with initial data $u_0\in BC(\mathbb{R}).$ Then $\ \Phi^\epsilon_t$ also has the properties (I-III). Moreover, for any given $t_0>0,$
$\Phi^\epsilon_{t_0}(u_0)< Q_{t_0}(u_0)$ provided $u_0(x)\geq 0$
and $\max|u_0|$ is small enough.
\end{lem}

\begin{proof}
We recall that $Q(x,t,u_0)$ is the solution of
$u_t=u_{xx}+\bar{b}(x)u(1-u)$ with initial data $u_0(x)$ and
$\Phi^\epsilon_{t}(u_0)$ is the solution of
$u^\epsilon_t=u^\epsilon_{xx}+(1-\epsilon)\bar{b}(x)u^\epsilon$
with initial data $u_0.$ Let $w=u-u^\epsilon.$ Then \[ \
w_t=w_{xx}+\bar{b}((1-u)u-(1-\epsilon)u^\epsilon).
\]
Since $t_0$ and $\epsilon$ are given, from the continuous
dependency of $u(x,t)$ on the initial data, if $\max u_0$ is
sufficiently small, then $u(x,t)<\epsilon,$ for $t\in[0,t_0].$

Then we have
\[ \
w_t=w_{xx}+\bar{b}((1-u)u-(1-\epsilon)u^\epsilon)>w_{xx}+\bar{b}(1-\epsilon)w.
\]
It can be rewritten as
\[
\ w_t>w_{xx}.
\]
Hence, $w=u-u^\epsilon\geq 0,$ for $x\in \mathbb R ,t\in[0,t_0].$
\end{proof}

\begin{proof}[Proof of Theorem \ref{cstarbequalcestarb}] To prove this theorem, we use the
results of  Weinberger \cite{s8} or the more abstract results of
Liang, Zhao \cite{s7}. In fact, Theorem \ref{them5.6} shows the
existence of $c^*(\bar{b})$ . Moreover, for any $t>0,$ we can
check that $\Phi_t$ satisfies the hypotheses of Theorem 2.5 in
\cite{s8} and for any $1>\epsilon>0,$ $\Phi^{\epsilon}_t$
satisfies the hypotheses of Theorem 2.4 in \cite{s8}. Moreover,
for any $\lambda\geq 0,$ the principal eigenvalue of
$u_{xx}-2\lambda u_x+(1-\epsilon)\bar{b}(x)u$ under the
periodicity conditions converges to the eigenvalue of
$u_{xx}-2\lambda u_x+\bar{b}(x)u$ under the same periodicity
conditions as $\epsilon\to 0.$ Hence we get
\[
c^*(\bar{b})=c_e^*(\bar{b}),\,\,\,\bar{b}\in\overline{\Lambda}(\alpha).\]
Similarly
\[
\tilde{c}^*(\bar{b})=\tilde{c}_e^*(\bar{b}),\,\,\,\bar{b}\in\overline{\Lambda}(\alpha).\]
To prove $c^*(\bar{b})=\tilde{c}^*(\bar{b}),$ we first prove that,
for smooth $b\in\Lambda(\alpha),$ $c^*(b)=\tilde{c}^*(b).$ Then by
using Proposition \ref{convergenceofeigenvalue} and the result
that $c^*(\bar{b})=c_e^*(\bar{b}),$ we get
$c^*(\bar{b})=\tilde{c}^*(\bar{b}).$

For smooth $b\in\Lambda(\alpha),$ consider the following two
operators: $-L_{\lambda,b}\psi=-\psi''+2\lambda\psi'-b\psi$ and
$-L_{-\lambda,b}\psi=-\psi''-2\lambda\psi'-b\psi.$ They are
adjoint operators, so they have the same principal eigenvalue
which means $\mu(\lambda,b)=\mu(-\lambda,b).$ By Definition
\ref{Defofcestar}, it holds that $c_e^*(b)=\tilde{c}_e^*(b).$
\end{proof}

%\subsection{Proof of $c^*(\bar{b})=\tilde{c}^*(\bar{b})$}
\subsection{Proof of Theorem \ref{cestarhlargerthancestarb} and
\ref{Themforspreadingspeed}}

Now we are ready to prove Theorem \ref{cestarhlargerthancestarb}
and \ref{Themforspreadingspeed}.

\begin{proof}[Proof of Theorem \ref{cestarhlargerthancestarb}]
By Lemma \ref{cestarhissupofcestarbbar},
\[
\
c_e^*(h)=\sup_{\bar{b}\in\overline{\Lambda}(\alpha)}c_e^*(\bar{b}).
\]
On the other hand, in the previous subsection, we have proven that
$c_e^*(\bar{b})=c^*(\bar{b})$ for
$\bar{b}\in\overline{\Lambda}(\alpha).$ Since
$h\in\overline{\Lambda}(\alpha),$ we get
\[
\
c^*(h)=\sup_{b\in\Lambda(\alpha)}c^*(b)=\max_{\bar{b}\in\overline{\Lambda}(\alpha)}c^*(\bar{b}).
\]

\end{proof}
\begin{proof}[Proof of Theorem \ref{Themforspreadingspeed}]
Theorem \ref{them5.6} implies that
\[
\
c^{**}(\bar{b})=c^*(\bar{b}),\,\,\,\tilde{c}^{**}(\bar{b})=\tilde{c}^*(\bar{b}).
\] By Theorem \ref{cstarbequalcestarb},
$c^*(\bar{b})=\tilde{c}^*(\bar{b}).$ Consequently,
\[
\ c^{**}(\bar{b})=\tilde{c}^{**}(\bar{b}).
\]
\end{proof}

\section{Proof of the lemmas}

In this section, we prove the following technical lemmas on the
equicontinuity of solutions of Cauchy problem
\eqref{equationinsection3}. These lemmas have been used in
Sections 3 and 5.

\begin{lem}\label{EquicontinuousForX}
Let $u(x,t,u_0,b)$ be the solution of \eqref{equationinsection3}
with $\bar{b}$ replaced by a smooth $b\in\Lambda(\alpha)$ with
initial data $u_0\in C(\mathbb R)\cap L^\infty(\mathbb R).$ Then
for any $\epsilon, \,M>0,$
$\{u(x,t,u_0,b)\}_{t\geq\epsilon,\|u_0\|\leq M,
b\in\Lambda(\alpha)} $ is uniformly equicontinuous in $x\in
\mathbb{R}$.
\end{lem}

\begin{lem}\label{EquicontinuousForT}
Let $u(x,t,u_0,b)$ be the solution of \eqref{equationinsection3}
with $\bar{b}$ replaced by a smooth $b\in\Lambda(\alpha)$ with
initial data $u_0\in C(\mathbb R)\cap L^\infty(\mathbb R).$ Then
for any $\epsilon, \,M>0,$ $\{u(x,t,u_0,b)\}_{x\in
\mathbb{R},\|u_0\|\leq M, b\in\Lambda(\alpha)} $  is uniformly
equicontinuous in $t\in [\epsilon,\infty)$.
\end{lem}

{}From Lemma \ref{EquicontinuousForX} and Lemma
\ref{EquicontinuousForT}, The following
lemma easily follows. We omit the proof since it is straightfoward.

\begin{lem}\label{Equicontinuity}
Let $u(x,t,u_0,b)$ be the solution of \eqref{equationinsection3}
with $\bar{b}$ replaced by a smooth $b\in\Lambda(\alpha)$ with
initial data $u_0\in C(\mathbb R)\cap L^\infty(\mathbb R).$ Then
for any $\epsilon, \,M>0,$ $\{u(x,t,u_0,b)\}_{\|u_0\|\leq M,
b\in\Lambda(\alpha)} $  is uniformly equicontinuous in $(x,t)\in
\mathbb{R}\times[\epsilon,\infty)$.
\end{lem}

\subsection{Proof of Lemma \ref{EquicontinuousForX} }
Given $b$ and $u_0$, we denote $u(x,t,u_0,b)$ by $u(x,t)$ simply.
First, we recall that $u(x,t)$ is a mild solution. It can be
written as
\[
\begin{split}
u(x,t)=&\frac{1}{\sqrt{4\pi t}}\int_\mathbb{R}
e^{-\frac{(x-y)^2}{4t}}u(y,0)dy\\&+\int_0^t\int_\mathbb{R}
\frac{e^{-\frac{(x-y)^2}{4(t-s)}}}{\sqrt{4\pi(t-s)}}b(y)u(y,s)(1-u(y,s))\,dyds).
\end{split}
\]

Let $t>0$ be given. In order to prove that $u(x,t)$ is
equicontinuous with respect to $x,$ we need to prove that for any
small $\epsilon>0,$ there exists a $\delta>0$ such that
\[
\ |u(x_1,t)-u(x_2,t)|\leq \epsilon,
\] if $|x_1-x_2|\leq\delta.$ For any $x_1>x_2 \in\mathbb R$,
\begin{eqnarray*}\displaystyle
 u(x_1,t)-u(x_2,t) = \frac{1}{\sqrt{4\pi t}}\int_\mathbb{R}
(e^{-\frac{(x_1-y)^2}{4t}}-e^{-\frac{(x_2-y)^2}{4t}})u(y,0)dy\\
+ \int_0^t\int_\mathbb{R}
\frac{e^{-\frac{(x_1-y)^2}{4(t-s)}}-e^{-\frac{(x_2-y)^2}{4(t-s)}}}{\sqrt{4\pi(t-s)}}b_n(y)u(y,s)(1-u(y,s))\,dyds.
\end{eqnarray*}

Next we prove that for any $\epsilon>0$, there exists a
$\delta>0$, such that the first part
\[ \ \frac{1}{\sqrt{4\pi t}}\int_\mathbb{R}
(e^{-\frac{(x_1-y)^2}{4t}}-e^{-\frac{(x_2-y)^2}{4t}})u(y,0)dy
\]
and the second part
\[
\ \int_0^t\int_\mathbb{R}
\frac{e^{-\frac{(x_1-y)^2}{4(t-s)}}-e^{-\frac{(x_2-y)^2}{4(t-s)}}}{\sqrt{4\pi(t-s)}}b_n(y)u(y,s)(1-u(y,s))\,dyds
\]
can be bounded by $\epsilon$ if $| x_1-x_2 |\leq \delta$.

For the first part,
\[
\begin{split}
 \frac{1}{\sqrt{4\pi t}}\int_\mathbb{R}
(e^{-\frac{(x_1-y)^2}{4t}}-e^{-\frac{(x_2-y)^2}{4t}})u(y,0)dy\\=\frac{1}{\sqrt{4\pi
t}}\int_\mathbb{R} -
\frac{2(x_{\xi,y}-y)}{4t}e^{-\frac{(x_{\xi,y}-y)^2}{4t}}&u(y,0)dy\times(x_1-x_2),
\end{split}
\]
where $x_{\xi,y}$ is a function of $y$ satisfying
\[
\ e^{-\frac{(x_1-y)^2}{4t}}-e^{-\frac{(x_2-y)^2}{4t}}=-
\frac{2(x_{\xi,y}-y)}{4t}e^{-\frac{(x_{\xi,y}-y)^2}{4t}}
\]
and $x_{\xi,y} \in (x_2,x_1).$ We know that the integration
\[
\ \int_\mathbb{R} | x| e^{-x^2}dx
\]
is bounded. Now we need to prove that
\[
\ \frac{1}{\sqrt{4\pi t}}\int_\mathbb{R} -
\frac{2(x_{\xi,y}-y)}{4t}e^{-\frac{(x_{\xi,y}-y)^2}{4t}}u(y,0)dy
\]
is bounded. Consider the following sets: \[ \ \Omega_i=\{y
\in\mathbb R| i\leq | \sup\{| y-x_1|,| y-x_2| \}\leq
i+1\},~i=1,2,3....
\] We have
\[
\begin{split}
 &\frac{1}{\sqrt{4\pi t}}\int_\mathbb{R} -
\frac{2(x_{\xi,y}-y)}{4t}e^{-\frac{(x_{\xi,y}-y)^2}{4t}}u(y,0)dy\\
&=\sum_{i\in Z} \int_{\Omega_i} \frac{-1}{\sqrt{4\pi t}\sqrt{4t}}
\frac{2(x_{\xi,y}-y)}{\sqrt{4t}}e^{-\frac{(x_{\xi,y}-y)^2}{4t}}u(y,0)dy.
\end{split}
\]
There exists a constant $M(t)>0$, such that for every $\Omega_i$,
\[
\begin{split}
\int_{\Omega_i} &| \frac{-1}{\sqrt{4\pi t}\sqrt{4t}}
\frac{2(x_{\xi,y}-y)}{\sqrt{4t}}e^{-\frac{(x_{\xi,y}-y)^2}{4t}}u(y,0)|
dy\\ &\leq M(t)\int_{\Omega_i}| y-x_1| e^{-(y-x_1)^2}dy
\end{split}
\]
which means the first part can be bounded by $\epsilon/2$ if
$|x_1-x_2|$ is small enough.

Next step is to prove that the second part
\[ \ |\int_0^t\int_\mathbb{R}
\frac{e^{-\frac{(x_1-y)^2}{4(t-s)}}-e^{-\frac{(x_2-y)^2}{4(t-s)}}}{\sqrt{4\pi(t-s)}}b(y)u(y,s)(1-u(y,s))\,dyds|
\leq \epsilon/2,
\] if $| x_1-x_2|\leq\delta.$

For notational simplicity, in what follows we write:
\[
\ \sup_{x,kL}g(x):=\sup_{x \in [kL,kL+L]}g(x).
\]

First, \[ \ \big|\int_0^t\int_\mathbb{R}
\frac{e^{-\frac{(x-y)^2}{4(t-s)}}}{\sqrt{4\pi(t-s)}}b(y)u(y,s)(1-u(y,s))\,dyds\big|
\]
\[
\ \leq M\int_0^t \frac{\alpha L}{\sqrt{4\pi(t-s)}}\sum_{k\in Z}
\sup_{y,kL}e^{-\frac{(x-y)^2}{4(t-s)}}\,ds=^{\exists} M'(t),
\] where $M>0$ is a constant related to the bound of $u$ and $M'(t)$ is bounded for fixed $t$. Hence
\[
\ \int_0^t\int_\mathbb{R}
\frac{e^{-\frac{(x-y)^2}{4(t-s)}}}{\sqrt{4\pi(t-s)}}b(y)u(y,s)(1-u(y,s))\,dyds
\]
is uniformly bounded for $x \in\mathbb R.$ Then we can choose a
$t^*<t$ such that \[ \ |\int_{t^*}^t\int_\mathbb{R}
\frac{e^{-\frac{(x_1-y)^2}{4(t-s)}}}{\sqrt{4\pi(t-s)}}b(y)u(y,s)(1-u(y,s))\,dyds|<\epsilon/4
\] and
\[ \ |\int_{t^*}^t\int_\mathbb{R}
\frac{e^{-\frac{(x_2-y)^2}{4(t-s)}}}{\sqrt{4\pi(t-s)}}b(y)u(y,s)(1-u(y,s))\,dyds|<\epsilon/4,
\] where $t^*$ is independent of the choice of $x_1$ and $x_2.$

Consider the integral from $0$ to $t^*$. \[ \
\int_0^{t^*}\int_\mathbb{R}
\frac{e^{-\frac{(x_1-y)^2}{4(t-s)}}-e^{-\frac{(x_2-y)^2}{4(t-s)}}}{\sqrt{4\pi(t-s)}}b(y)u(y,s)(1-u(y,s))\,dyds\]
\[ \ =\int_0^{t^*}\int_\mathbb{R}
\frac{-2\frac{x_{\xi,y}-y}{4(t-s)}e^{-\frac{(x_{\xi,y}-y)^2}{4(t-s)}}}{\sqrt{4\pi(t-s)}}b(y)u(y,s)(1-u(y,s))\,dyds\,(x_1-x_2)
\]
\[ \
=\int_0^{t^*}\frac{-2}{\sqrt{4\pi(t-s)}\sqrt{4(t-s)}}\int_\mathbb{R}
\frac{\frac{x_{\xi,y}-y}{\sqrt{4(t-s)}}e^{-\frac{(x_{\xi,y}-y)^2}{4(t-s)}}}{\sqrt{4\pi(t-s)}}b(y)u(y,s)(1-u(y,s))\,dyds
\]
\[
\ \times (x_1-x_2),
\]
where $x_{\xi,y}$ is a function of $y$ and takes values between
$x_2$ and $x_1.$

Since $t^*<t$, we can get \[ \
\big|\int_0^{t^*}\frac{-2}{\sqrt{4\pi(t-s)}\sqrt{4(t-s)}}\int_\mathbb{R}
\frac{\frac{x_{\xi,y}-y}{\sqrt{4(t-s)}}e^{-\frac{(x_{\xi,y}-y)^2}{4(t-s)}}}{\sqrt{4\pi(t-s)}}b(y)u(y,s)(1-u(y,s))\,dyds\big|
\] is bounded.

Then if we choose $\delta$ small enough, when $| x_1-x_2|\leq
\delta$, we have
\[
\begin{split}
\big|\int_0^{t^*}\frac{-2}{\sqrt{16\pi}(t-s)}&\int_\mathbb{R}
\frac{\frac{x_{\xi,y}-y}{\sqrt{4(t-s)}}e^{-\frac{(x_{\xi,y}-y)^2}{4(t-s)}}}{\sqrt{4\pi(t-s)}}b(y)u(y,s)(1-u(y,s))\,dyds\\
&\times (x_1-x_2)\big|\leq \epsilon/2.
\end{split}
\]
It is not difficult to see that $\delta$ is independent of $x_1$
and $x_2.$

Combining the above estimates, we see that there exists $\delta>0$
such that if $| x_1-x_2| \leq \delta$, then  $|
v_n(x_1,t)-v_n(x_2,t)|\leq \epsilon.$ It is not difficult to see
that the above estimate is independent of the choice of $b$ and
$u_0.$ This completes the proof of the lemma.

\subsection{Proof of Lemma \ref{EquicontinuousForT}
}
As in lemma \ref{EquicontinuousForX}, we still use
\[
\begin{split} u(x,t)&=\frac{1}{\sqrt{4\pi t}}\int_\mathbb{R}
e^{-\frac{(x-y)^2}{4t}}u(y,0)dy\\
 &+\int_0^t\int_\mathbb{R}
\frac{e^{-\frac{(x-y)^2}{4(t-s)}}}{\sqrt{4\pi(t-s)}}b(y)u(y,s)(1-u(y,s))\,dyds.
\end{split}
\]

Here we consider an interval $[T_1,T_2]\subset\mathbb R$, where
$T_2>T_1>0.$ For $t_1<t_2 \in [T_1,T_2]$,
\[
\begin{split}
&u(x,t_1)-u(x,t_2)\\=&\frac{1}{\sqrt{4\pi t_1}}\int_\mathbb{R}
e^{-\frac{(x-y)^2}{4t_1}}u(y,0)dy-\frac{1}{\sqrt{4\pi
t_2}}\int_\mathbb{R} e^{-\frac{(x-y)^2}{4t_2}}u(y,0)dy
\\
&+\int_0^{t_1}\int_\mathbb{R}\frac{e^{-\frac{(x-y)^2}{4(t_1-s)}}}{\sqrt{4\pi(t_1-s)}}b(y)u(y,s)(1-u(y,s))\,dyds
\\
&-\int_0^{t_2}\int_\mathbb{R}\frac{e^{-\frac{(x-y)^2}{4(t_2-s)}}}{\sqrt{4\pi(t_2-s)}}b(y)u(y,s)(1-u(y,s))\,dyds.
\end{split}
\]
We divide the right-hand side into two parts as in Lemma
\ref{EquicontinuousForX}.  \[ \ (I):=\frac{1}{\sqrt{4\pi
t_1}}\int_\mathbb{R}
e^{-\frac{(x-y)^2}{4t_1}}u(y,0)dy-\frac{1}{\sqrt{4\pi
t_2}}\int_\mathbb{R} e^{-\frac{(x-y)^2}{4t_2}}u(y,0)dy
\] and
\[
\
 (II):=\int_0^{t_1}\int_\mathbb{R}\frac{e^{-\frac{(x-y)^2}{4(t_1-s)}}}{\sqrt{4\pi(t_1-s)}}b(y)u(y,s)(1-u(y,s))\,dyds
\]
\[
\
-\int_0^{t_2}\int_\mathbb{R}\frac{e^{-\frac{(x-y)^2}{4(t_2-s)}}}{\sqrt{4\pi(t_2-s)}}b(y)u(y,s)(1-u(y,s))\,dyds.
\] First,
\[
\ (I)=\int_\mathbb{R}
\frac{\frac{(x-y)^2}{4t_{\xi,y}^2}e^{-\frac{(x-y)^2}{4t_{\xi,y}}}-\frac{4\pi}{2\sqrt{4\pi
t_{\xi,y}}}e^{-\frac{(x-y)^2}{4t_{\xi,y}}}}{4\pi
t_{\xi,y}}(t_1-t_2)u(y,0)dy,
\]
where $t_{\xi,y}$ is between $t_2$ and $t_1$.

As discussed in Lemma \ref{EquicontinuousForX}, we may get that
$(I)$ can be very small if $|t_1-t_2|$ is small uniformly in
$x\in \mathbb R.$

The next step is to prove that for any $\epsilon>0$, there exists
a $\delta>0$ such that the absolute value of the term $(II)$ is
less than or equal to $\epsilon$ if $|t_1-t_2|\leq\delta$.
\[
\begin{split}
(II)=&\int_0^{t_1}\int_\mathbb{R}\frac{e^{-\frac{(x-y)^2}{4(t_1-s)}}}{\sqrt{4\pi(t_1-s)}}b(y)u(s,y)(1-u(y,s))\,dyds\\
&-\int_0^{t_2}\int_\mathbb{R}\frac{e^{-\frac{(x-y)^2}{4(t_2-s)}}}{\sqrt{4\pi(t_2-s)}}b(y)u(y,s)(1-u(y,s))\,dyds\\
=&\int_0^{t_1}\int_\mathbb{R}\frac{e^{-\frac{(x-y)^2}{4(t_1-s)}}}{\sqrt{4\pi(t_1-s)}}b(y)u(y,s)(1-u(y,s))\,dyds\\
&-\int_0^{t_1}\int_\mathbb{R}\frac{e^{-\frac{(x-y)^2}{4(t_2-s)}}}{\sqrt{4\pi(t_2-s)}}b(y)u(y,s)(1-u(y,s))\,dyds\\
&-\int_{t_1}^{t_2}\int_\mathbb{R}\frac{e^{-\frac{(x-y)^2}{4(t_2-s)}}}{\sqrt{4\pi(t_2-s)}}b(y)u(y,s)(1-u(y,s))\,dyds.
\end{split}
\]
It is not difficult to show that
\[
\
\Big|\int_{t_1}^{t_2}\int_\mathbb{R}\frac{e^{-\frac{(x-y)^2}{4(t_2-s)}}}{\sqrt{4\pi(t_2-s)}}b(y)u(y,s)(1-u(y,s))\,dyds\Big|\leq
D\sqrt{| t_1-t_2|},
\] where $D$ is a constant. Therefore it suffices to consider
\[
\begin{split}
\int_0^{t_1}\int_\mathbb{R}\frac{e^{-\frac{(x-y)^2}{4(t_1-s)}}}{\sqrt{4\pi(t_1-s)}}b(y)u(y,s)(1-u(y,s))\,dyds\hspace{30pt}\\
\hspace{20pt}-\int_0^{t_1}\int_\mathbb{R}\frac{e^{-\frac{(x-y)^2}{4(t_2-s)}}}{\sqrt{4\pi(t_2-s)}}b(y)u(y,s)(1-u(y,s))\,dyds.
\end{split}
\]
Once again, as in lemma \ref{EquicontinuousForX}, for
$\epsilon>0$, there exists a $t^*<t_1$, such that \[ \
\big|\int_{t^*}^{t_1}\int_\mathbb{R}\frac{e^{-\frac{(x-y)^2}{4(t_1-s)}}}{\sqrt{4\pi(t_1-s)}}b(y)u(y,s)(1-u(y,s))\,dyds\big|\leq
\epsilon/4
\] and
\[
\
\big|\int_{t^*}^{t_1}\int_\mathbb{R}\frac{e^{-\frac{(x-y)^2}{4(t_2-s)}}}{\sqrt{4\pi(t_2-s)}}b(y)u(y,s)(1-u(y,s))\,dyds\big|\leq
\epsilon/4.
\]
The left-hand side is an integral from $0$ to $t^*$ which is
\[
\begin{split}
&\int_0^{t^*}\int_\mathbb{R}(\frac{e^{-\frac{(x-y)^2}{4(t_1-s)}}}{\sqrt{4\pi(t_1-s)}}-\frac{e^{-\frac{(x-y)^2}{4(t_2-s)}}}{\sqrt{4\pi(t_2-s)}})b(y)u(y,s)(1-u(y,s))\,dyds\\
 =&\int_0^{t^*}\int_\mathbb{R}
\frac{\frac{(x-y)^2}{4(t_{\xi,y}-s)^2\sqrt{4\pi(t_{\xi,y}-s)}}e^{-\frac{(x-y)^2}{4(t_{\xi,y}-s)}}-\frac{4\pi}{2\sqrt{4\pi(t_{\xi,y}-s)}}e^{-\frac{(x-y)^2}{4(t_{\xi,y}-s)}}}
{4\pi(t_{\xi,y}-s)}\\ &\hspace{40pt}\times
(t_1-t_2)b(y)u(y,s)(1-u(y,s))\,dyds\\
 =&\int_0^{t^*}\int_\mathbb{R}
\frac{\frac{(x-y)^2}{4(t_{\xi,y}-s)^2\sqrt{4\pi(t_{\xi,y}-s)}}e^{-\frac{(x-y)^2}{4(t_{\xi,y}-s)}}}{4\pi(t_{\xi,y}-s)}(t_1-t_2)b(y)u(y,s)(1-u(y,s))\,dyds\\
 &-\int_0^{t^*}\int_\mathbb{R}
\frac{\frac{4\pi}{2\sqrt{4\pi(t_{\xi,y}-s)}}e^{-\frac{(x-y)^2}{4(t_{\xi,y}-s)}}}{4\pi(t_{\xi,y}-s)}(t_1-t_2)b(y)u(y,s)(1-u(y,s))\,dyds.
\end{split}
\]
Next we prove that both
\[
\ \int_0^{t^*}\int_\mathbb{R}
\frac{\frac{(x-y)^2}{4(t_{\xi,y}-s)^2\sqrt{4\pi(t_{\xi,y}-s)}}e^{-\frac{(x-y)^2}{4(t_{\xi,y}-s)}}}{4\pi(t_{\xi,y}-s)}b(y)u(y,s)(1-u(y,s))\,dyds
\]
and
\[
\ -\int_0^{t^*}\int_\mathbb{R}
\frac{\frac{4\pi}{2\sqrt{4\pi(t_{\xi,y}-s)}}e^{-\frac{(x-y)^2}{4(t_{\xi,y}-s)}}}{4\pi(t_{\xi,y}-s)}b(y)u(y,s)(1-u(y,s))\,dyds
\]
are uniformly bounded in $x\in\mathbb R.$ Since they are similar,
we prove the former. We have
\[
\begin{split}
&\big|\int_0^{t^*}\int_\mathbb{R}
\frac{\frac{(x-y)^2}{4(t_{\xi,y}-s)^2\sqrt{4\pi(t_{\xi,y}-s)}}e^{-\frac{(x-y)^2}{4(t_{\xi,y}-s)}}}{4\pi(t_{\xi,y}-s)}b(y)u(y,s)(1-u(y,s))\,dyds\big|\\
&\leq M\big|\int_0^{t^*}\frac{1}{8\sqrt{\pi}
(t_1-s)^{7/2}}\int_\mathbb{R}
{(x-y)^2}e^{-\frac{(x-y)^2}{4(t_{\xi,y}-s)}}b(y)\,dyds\big|
\end{split}
\]
As we have proven above, \[ \ \int_\mathbb{R}
{(x-y)^2}e^{-\frac{(x-y)^2}{4(t_{\xi,y}-s)}}b(y)dy
\]
is uniformly bounded for any $x \in \mathbb R$. Since $t^*<t_1$,
this integral is bounded by a constant $M(t^*)>0$. Consequently
there exists a $\delta_1>0$, such that if $t_2-t_1\leq\delta_1$,
then \[ \ \big|\int_0^{t^*}\int_\mathbb{R}
\frac{\frac{(x-y)^2}{4(t_{\xi,y}-s)^2\sqrt{4\pi(t_{\xi,y}-s)}}e^{-\frac{(x-y)^2}{4(t_{\xi,y}-s)}}}{4\pi(t_{\xi,y}-s)}(t_1-t_2)b(y)u(y,s)(1-u(y,s))\,dyds
\big|\leq \epsilon/4. \] We can also prove that there exists $
\delta_2>0$ such that if $t_2-t_1\leq \delta_2$, then
\[
\ \big|\int_0^{t^*}\int_\mathbb{R}
\frac{\frac{4\pi}{2\sqrt{4\pi(t_{\xi,y}-s)}}e^{-\frac{(x-y)^2}{4(t_{\xi,y}-s)}}}{4\pi(t_{\xi,y}-s)}(t_1-t_2)b(y)u(y,s)(1-u(y,s))\,dyds\big|\leq
\epsilon/4.
\]
Hence, for any $\epsilon>0$, there exists a constant $\delta>0$
such that
\[
| u(x,t_1)-u(x,t_2)|\leq \epsilon \hbox{ if } | t_1-t_2|\leq
\delta.
\]
Here $t_1, t_2 \in [T_1,T_2]$ and the equicontinuity is
independent of the choice of $x$, $u_0$ and $b$. The proof of the
lemma is complete.

%%%%%%%%%%%%%%%%%%%%%%%%%%%%%%%%%%%%%%%%%%%%%%%%%%%%%%%%%%%%%%%%%%%%%%%%%%%%%%%%%%%%%%OVER

\end{document}